\documentclass[a4paper]{article}

\usepackage{a4wide}
\usepackage{amsmath,amssymb,latexsym,cite,amsthm,hyperref, multicol, bussproofs}
\usepackage{type1cm}      	
\usepackage{microtype}
\usepackage{newtxtext}       	%
\usepackage{newtxmath}       	
\usepackage[english]{babel}

\newcommand\bcmdtab{\noindent\bgroup\tabcolsep=0pt%
  \begin{tabular}{@{}p{10pc}@{}p{20pc}@{}}}
\newcommand\ecmdtab{\end{tabular}\egroup}

 \newtheorem{theorem}{Theorem}[section]
 \newtheorem{proposition}[theorem]{Proposition}
 \newtheorem{lemma}[theorem]{Lemma}

\theoremstyle{definition}
 \newtheorem{definition}[theorem]{Definition}
 \newtheorem{remark}[theorem]{Remark}

 \newtheorem{example}[theorem]{Example}

\usepackage{xcolor}
\colorlet{darkmagenta}{magenta!85!black}
\colorlet{darkgreen}{green!55!black}
\colorlet{darkblue}{blue!85!black}
\colorlet{darkred}{red!80!black}

\newcommand{\condition}[1]{{\rm\textsc{#1}}}
\newcommand{\REFL}{\condition{Reflexivity}}
\newcommand{\GENREFL}{\condition{Generalized\ Reflexivity}}
\newcommand{\CUT}{\condition{Cut}}
\newcommand{\RELCUT}{\condition{Relevant\ Cut}}

\newcommand{\CUMCUT}{\condition{Cumulative\ Cut}}
\newcommand{\MONO}{\condition{Monotonicity}}
\newcommand{\MULTICUT}{\condition{Multi-Cut}}
\newcommand{\CONT}{\condition{Contraction}}
\newcommand{\FINI}{\condition{Finitarity}}
\newcommand{\THMREM}{\condition{Theorem\ Removal}}
\newcommand{\TRAN}{\condition{Transitivity}}
\newcommand{\TRASH}{\condition{Theorem Reflexivity}}
\newcommand{\COMP}{\condition{Compatibility}}

\newcommand{\node}[1]{\ensuremath{{\boldsymbol{#1}}}} 

\DeclareMathOperator{\MPow}{\mathcal{P}^{\#}}
\DeclareMathOperator{\FMPow}{\mathcal{P}_{\!\mathit{fin}}^\#}

\newcommand{\Msetminus}{\ensuremath{\mathbin{\dot-}}}

\usepackage{tikz-cd}
\usepackage[shortcuts]{extdash}

\newcommand{\alg}[1]{{\ensuremath{\boldsymbol {\mathit{#1}}}}}

{\rm \begin{trivlist}\item\begin{tabular}{l@{\quad}l@{\qquad}l}}
    {\end{tabular}\end{trivlist}}

\newcommand{\tuple}[1]{\ensuremath{\langle{#1}\rangle}}  
\newcommand{\vectn}[1]{\ensuremath{#1_1,\ldots,#1_n}}    
\newcommand{\vectk}[1]{\ensuremath{#1_1,\ldots,#1_k}}    

\newcommand{\class}[2]{\{#1\mid #2\}}
\newcommand{\fclass}[1]{\{#1\}}

\newcommand{\numeral}[1]{\ensuremath{\bar{#1}}}
\newcommand{\Numeral}[1]{\ensuremath{\overline{#1}}} 
\newcommand{\1}{\numeral{1}}
\newcommand{\0}{\numeral{0}}

\newcommand{\form}{\mathit{Fm}}                   

\newcommand{\f}{\ensuremath{\varphi}}
\newcommand{\p}{\ensuremath{\psi}}
\newcommand{\x}{\ensuremath{\chi}}
\newcommand{\AS}{\ensuremath{\mathcal{AS}}}

\newcommand{\hairspace}{\ifmmode\mskip1mu\else\kern0.08em\fi}
\newcommand{\emptymultiset}{\ensuremath{[\hairspace]}}

\newcommand{\logic}[1]{\ensuremath{\mathbf{#1}}}
\newcommand{\R}{\logic R}
\newcommand{\E}{\logic E}
\newcommand{\T}{\logic T}
\newcommand{\RM}{\logic{RM}}
\newcommand{\CL}{\logic{CL}}
\newcommand{\BCI}{\logic{BCI}}

\let\models\vDash

\let\leq\leqslant
\let\geq\geqslant

\makeatletter
\newcommand{\bigplus}{%
  \DOTSB\mathop{\mathpalette\mattos@bigplus\relax}\slimits@
}
\newcommand\mattos@bigplus[2]{%
  \vcenter{\hbox{%
    \sbox\z@{$#1\sum$}%
    \resizebox{!}{0.9\dimexpr\ht\z@+\dp\z@}{\raisebox{\depth}{$\m@th#1+$}}%
  }}%
  \vphantom{\sum}%
}
\makeatother

\begin{document}

\begin{center}

{\Large \bf Relevant Consequence Relations: An Invitation }

\

            GUILLERMO BADIA\\
            University of Queensland\\
            \url{https://orcid.org/0000-0002-5597-6794}\\[1ex]
            and\\[1ex]
            LIBOR B{\v{E}}HOUNEK\\
            University of Ostrava\\
            \url{https://orcid.org/0000-0001-8570-9657}\\[1ex]
            and \\[1ex]
            PETR CINTULA\\
            Institute of Computer Science of the Czech Academy of Sciences\\
            \url{https://orcid.org/0000-0002-3617-1392}\\[1ex]
            and \\[1ex]
            ANDREW TEDDER\\
            Ruhr University Bochum\\
            \url{https://orcid.org/0000-0002-2303-003X}

\end{center}

\paragraph*{Abstract.}
{\small We generalize the notion of \emph{consequence relation} standard in abstract treatments of logic to accommodate intuitions of \emph{relevance}. The guiding idea follows the \emph{use criterion}, according to which in order for some premises to have some conclusion(s) as consequence(s), the premises must each be \emph{used} in some way to obtain the conclusion(s). This relevance intuition turns out to require not just a failure of monotonicity, but also a move to considering consequence relations as obtaining between \emph{multisets}. We motivate and state basic definitions of relevant consequence relations, both in single conclusion (asymmetric) and multiple conclusion (symmetric) settings, as well as derivations and theories, guided by the use intuitions, and prove a number of results indicating that the definitions capture the desired results (at least in many cases).}

\paragraph*{Keywords:}
relevant entailment, substructural logics, multiset consequence relations, multiple conclusions

\paragraph*{MSC~2020:}
Primary: 03B22 \;\;Secondary: 03A05, 03B47

\section{Introduction}
\label{sect:Intro}

Logical consequence is the central notion of the logical enterprise. Meta-mathematical studies of consequence relations start with the works of Alfred Tarski \cite{tarski1936} and are part of a long and distinguished tradition ever since 
(e.g., \cite{Rasiowa,BP86,DunnH}). The standard mathematical rendering of the notion of consequence, however, involves many structural presuppositions that make it unsuitable for a natural representation of some well known non-classical logics, such as substructural logics (e.g., \cite{SHD,restall2000,GJKO}). The failure of the Tarskian approach to capture the so-called `internal consequence relation' of substructural logics is a well known fact (see, e.g., \cite[p.~78]{GJKO}) that provides reasons to investigate alternate definitions of consequence.

Prominent among substructural logics are \emph{relevant} logics (e.g., \cite{AB, AndersonBelnapDunn1992, Routleyetal1982}), which can be broadly characterised as those which seek to enforce a relevance constraint on logical consequence, according to which in order for some formula(s) to be entailed by some other formula(s) (or to be ``a logical consequence of'' them), the entailed must be relevant to the entailing. There have been a handful of proposed formalizations of the relevance constraint; perhaps the most famous is the \emph{Variable Sharing} criterion, which requires that in order for an entailment claim to be logically correct, the premises and the conclusion must have some propositional variable in common (see \cite[\S22.1.3]{AB}). Another approach, not quite as venerable, though older, is embodied in the so-called Use Criterion:

\begin{quote}
\emph{Use Criterion}:
A formula $\f$ is relevantly deducible from a collection of formulas $\Gamma$ just in case there is a derivation of $\f$ from $\Gamma$ in which each $\gamma\in\Gamma$ is \emph{used}.
\end{quote}

This statement is adapted from \cite{Meyer74}, but a version is already mentioned, in passing, in Church's early work \cite[Fact 20]{Church51} and was further developed in the Fitch-style natural deduction systems presented in \cite{AB}. The idea here is, in certain respects, a bit simpler than the variable sharing property and is, in addition, more generalizable. Given \emph{any} system of derivation for any logic, even one which violates the variable sharing criterion, we can consider a definition of logical consequence which obeys the use criterion and is, in that sense and to that extent, relevant. It is relevant consequence relations of this kind that we shall investigate in this paper, in a more abstract fashion than has been done before. Usually in relevant logic circles, logics are understood as sets of formulas, and logical consequence has been taken to be expressed by the theoremhood of implication formulas (as in the ``internal'' consequence relations mentioned above). While metatheoretical characterisations of  so-called \emph{external} relevant consequence relations have come up for discussion (e.g., \cite{Meyerdiss}, \cite[\S22.2, \S23.6]{AB}, \cite{DunnH}), this work hasn't been done in the kind of abstractness and generality that work concerning Tarskian consequence relations has been. Our aim in this paper is to begin to fill the gap. 

One obvious result of the use-criterion approach to defining relevant consequence relations is that we don't have any guarantee that the resulting relations will be \emph{monotone}. That is, while we may have that $\Gamma$ entails~$\alpha$, because there is a derivation of $\alpha$ in which each formula in~$\Gamma$ is used, we may not have such a derivation for any arbitrary $\Gamma'$ containing~$\Gamma$. Such failures of monotonicity indicate the key respect in which relevant consequence relations, in our sense, differ from Tarskian relations, and provide some mathematical motivation for investigating such relations. 

The failure of monotonicity is related to failures of various \emph{weakening} axioms and rules in standard relevant logics. In the internal consequence relations of relevant logics, one has a choice of how to combine premises: basically, the choice is between using lattice conjunction $\land$ or fusion (``multiplicative conjunction'')~$\circ$. If we use the former to combine premises, then, in general, we don't have a deduction theorem, as $\rightarrow$ is not assumed to residuate $\land$ (indeed, $\f\land\p\rightarrow\chi$ does not imply $\f\rightarrow(\p\rightarrow\chi)$), whereas using the latter we retain residuation while losing monotonicity (as $(\f\circ\p)\rightarrow\f$ is not generally valid). Our approach is closer in spirit to the latter (though with some complications), hence the failure of monotonicity.

It should be noted that mononticity has been questioned on other grounds, perhaps most famously in the context of \emph{non-monotonic} logics, as studied by Makinson \cite{MakinsonD} and others. These systems, and their motivations, are rather different from relevant logics -- having to do with forms of \emph{defeasible} reasoning, and not involving a \emph{revisionist} critique of Tarskian consequence, as in relevant logic -- so we'll not dwell on these systems beyond briefly mentioning them to note the salient differences. 

Since the multiplicative conjunction of relevant logics is non-idempotent (in particular, we usually want to avoid having $(\f\circ\f)\rightarrow\f$), we will generally need to work with \emph{multisets} of premises -- i.e., collections in which each element can occur multiple times -- rather than ordinary sets. Even if the use of multisets can be viewed as a mere artifact of our approach to relevant consequence, multiset consequence relations are of independent interest and possess a rich mathematical theory. Monotonic multiset consequence relations have already been studied in an abstract setting~\cite{cintula2019}; in this paper, we will limit our investigation of (possibly non-monotonic) multiset consequence relations to the relevant ones, and will leave a more general study for future work.

Finally, we'll consider relevant consequence relations both with single and multiple conclusions. In the single conclusion setting, things are simpler and we obtain a collection of results characterising our target relations which flesh out the way our approach captures aspects of the intended interpretation. In the multiple conclusion settings, things get a bit more complex, and our approach is a little different than usual. A consecution in a multiple-conclusion consequence relation, $\Gamma \vdash \Delta$, is  typically taken to imply that from all the elements of~$\Gamma$, at least one of the elements of~$\Delta$ (or a disjunction of finitely many of them) follows. This might be called the \emph{disjunctive reading} of a multiple-conclusion consequence relation, and constitutes the dominant approach in the literature. There is, however, another, and perhaps more straightforward, reading which we call the \emph{conjunctive reading:} from all the elements of $\Gamma$, all the elements of $\Delta$ follow. Clearly, in logical settings with the contraction rule (where $\Gamma$ and $\Delta$ are sets), this can be reduced to a series of elementary claims $\Gamma\vdash \delta$ for each $\delta \in \Delta$; this may explain why disjunctive readings are more prevalent in the literature. Nevertheless, the conjunctive reading has also been used in classical settings, e.g., in the abstract category-theoretical study of consequence \cite{GT}; and perhaps surprisingly, the conjunctive reading can be traced back to the very origins of modern mathematical logic. In fact, as early as in the first half of 1800's, Bernard Bolzano \cite[p.~54]{bolzano2004} introduced a notion of multiple-conclusion consequence with the conjunctive reading:\looseness-1

\begin{quote}
One especially noteworthy case occurs, however, if not just some, but all of the ideas that, when substituted for $i, j, \dots$ in $A, B, C, \dots$ make all these true, also make all of $M, N, O, \dots$ true [$\dots$] with respect to the variable parts $i, j, \dots$
\end{quote}

For reasons we'll discuss later, the conjunctive reading is a nice fit for the study of multiple-conclusion relevant consequence relations, so that will be our approach here.\footnote{
    This can be contrasted with a different approach, which takes the disjunctive reading but uses \emph{multiplicative} disjunction, or \emph{fission}, as is available, for instance, in the relevant logic \textbf{R}. We'll come back to this point later; but see \cite[\S6.15]{DunnH} for some further discussion.}
For the sake of simplicity, we'll restrict our attention to \emph{finitary} consequence relations throughout this paper. A generalization to infinite multisets of premises or conclusions turns out to present certain technical challenges that we prefer to avoid in this initial exploration and leave them for future research.

Our work builds on previous work by Avron \cite{Avron, Av92}, who has investigated a range of multiset consequence relations, including non-monotonic ones, especially in the case of linear logic \cite{avron1988}. In our work, we focus particularly on relevance (via the use criterion), employ a conjunctive reading of the conclusions in the multi-conclusion setting, and abstract away from the behaviour of particular kinds of connectives. Nonetheless, our debt to Arnon's work in this area should be noted.

This article will be arranged as follows: \S2 will go into detail about why we employ multisets, and why doing so is called for by considerations of relevance in particular, by considering some untoward consequences of the standard, Tarskian, account of logical consequence. In \S3, we shall introduce the key definitions for single-conclusion consequence: namely that of a (relevant) tree proof in an axiomatically defined logic and, from there, the abstract definition of single-conclusion consequence (a form of non-monotonic consequence). Furthermore, we'll prove some results indicating that the abstract notion nicely captures the important features of our concrete starting point. \S4 considers some examples of the system in practice, particularly to some extensions of the substructural implication logic \BCI, including some relevant logics (and indicates some limitations of the approach for capturing relevant logics in languages including their \emph{extensional} vocabulary). In \S5 we extend the framework of \S3 to accommodate multiple-conclusion consequence relations, proving some further results indicating that our proposed definitions work as intended. Finally, in \S6, we consider in some detail the question of carrying out derivations in the multiple-conclusion case, and, in \S7, propose a way to define \emph{theory}, adapting the standard definition to our setting, which does some of the work one expects of that concept.

\section{Tarskian Consequence Relations; from Sets to Multisets}
\label{sect:Tarskian}

We'll recall the standard, Tarskian, definition of consequence relation, and consider reasons to alter certain key features of it. Let us fix a set $\form$ of formulas. Unless said otherwise, we do not assume that formulas have any internal structure.\footnote{Thus, even though we only give propositional examples throughout this paper, the definitions and results apply equally well to (relevant-flavoured) predicate or modal substructural logics.}

\begin{definition}[Finitary Tarskian Consequence Relations]\label{def:Tarski}
A \emph{finitary Tarskian consequence relation} is a relation $\vdash$ between \emph{finite} sets of elements of $\form$ and elements of $\form$ satisfying the following constraints: 

\begin{itemize}
\itemsep=0em
    \item $\f\vdash\f$ \hfill \REFL 
    \item If $\Gamma,\p\vdash\f$ and $\Delta\vdash\p$ then $\Gamma,\Delta\vdash\f$ \hfill \CUT 
    \item If $\Gamma\vdash\f$ then $\Gamma,\p\vdash\f$ \hfill \MONO 
\end{itemize}



\end{definition}

\begin{remark}\label{r:TarskiCanBeRelevant}
Traditionally, finitary Tarskian consequence relations $\vdash$ are defined as relations between \emph{arbitrary sets} of formulas and formulas obeying the variants of \REFL, \CUT\ and \MONO\ (for arbitrary sets~$\Gamma$ and~$\Delta$) plus the condition: 
\begin{itemize}
\item If\/ $\Gamma\vdash \f$, then there is finite $\Delta\subseteq \Gamma$ such that $\Delta\vdash \f$. \hfill \FINI
\end{itemize}
It is easy to see that (due to  \MONO)
there is one-one correspondence between these two definitions. 


\pagebreak
 
The use of infinite sets of premises has numerous technical advantages in the Tarskian setting which as we will see later, unfortunately, mostly disappear in the relevant and multiset setting. For reasons of simplicity and naturalness we'll restrict our attention to finite collections of premises throughout the paper. On the mathematical side, this assumption makes our definitions simpler and allows us to avoid some difficulties.
\end{remark}

We'll be concerned throughout with axiomatic presentations of logics, and we'll wind up constructing them out of \emph{consecutions}, but before doing so in formal detail, let's just consider an informal example to note some properties of axiomatic derivations in a tree form. Consider the following proof in classical logic (\CL) of the fact that $\f\rightarrow({\p\rightarrow\x}),{\f\land\p}\vdash_{\CL}\x$, using a tacit, but fairly standard axiomatic presentation of classical (or indeed intuitionistic) logic in the $\{\land,\rightarrow\}$-fragment:

\begin{center}
\AxiomC{$\f\land\p$}
\AxiomC{$\f\land\p\rightarrow\f$}
\BinaryInfC{$\f$}
\AxiomC{$\f\rightarrow(\p\rightarrow\x)$}
\BinaryInfC{$\p\rightarrow\x$}
\AxiomC{$\f\land\p$}
\AxiomC{$\f\land\p\rightarrow\p$}
\BinaryInfC{$\p$}
\BinaryInfC{$\x$}
\DisplayProof
\end{center}

\noindent
Notice that the structure relies on using the deduction rule of \emph{modus ponens}, and that the premise $\f\land\p$ is used twice. In a standard axiomatic presentation (in the \textsc{Fmla} framework, to use Humberstone's terminology \cite{Humberstone11}), axioms are taken to be formulas and deduction rules to operate thereon. We'll employ a slightly different presentation by dealing with \emph{consecutions}: these are objects of the form $\Gamma\rhd\f$ where $\Gamma$ is a finite set of formulas and $\f$ a formula.
\emph{Axioms,} then, are consecutions of the form $\varnothing\rhd\f$, and all the other consecutions are \emph{deduction rules:} for instance, we would have the axiom $\varnothing\rhd\f\land\p\rightarrow\f$ and the deduction rule $\f\rightarrow\p,\f\rhd\p$ in a consecution-version of the above proof tree. 

\begin{remark}

We take proofs to be finite trees rather than finite sequences. The latter choice is more standard, and the choice is somewhat arbitrary. Having said that, we'll proceed with tree proofs as they simplify some parts of the presentation once we start incorporating relevance considerations.

\end{remark}

\begin{definition}[Tree-proof]\label{d:tree-proof-sets}
Let $\AS$ be an axiomatic systems, i.e., a collection of axioms and deduction rules. In what follows, abusing notation, we identify any axiom of the form $\emptyset\rhd\varphi$ with the formula $\varphi$. A \emph{proof of a formula $\f$ from a set of premises $\Gamma$} in $\AS$ is a finite tree labeled by formulas such that:
\begin{enumerate}
\item The root is labelled by $\f$.
\item A formula labeling a leaf is an axiom of~\AS\ or an element of $\Gamma$.
\item For each non-leaf node $\node n$, \AS\ contains a rule $\Delta \rhd \p$ such that $\p$ is the label of $\node n$ and $\Delta$ is the set of formulas labeling the immediate predecessors of $\node n$.
\end{enumerate}
If there is a \emph{tree-proof} of a formula $\varphi$ in $\AS$ from a set of premises $\Gamma$, we write $\Gamma\vdash_{\!\AS} \f$.
\end{definition}

This is a standard definition, but in adapting it too swiftly to a non-classical setting, we accidentally introduce some problems; of particular focus here are problems of \emph{relevance} and problems concerning \emph{multiple uses} of premises. For the first, clearly any Tarskian consequence relation will allow for valid arguments in which some premises are irrelevant, and not needed: for instance, if we have $\varnothing\rhd\f\rightarrow\f$ among the axioms of $\AS$, then we'll have $\p\vdash_{\AS}\f\rightarrow\f$, despite the fact that $\p,\f$ may have nothing to do with one another. Stated in terms of the use criterion, there is, at best, only a rather attenuated sense in which $\p$ can be said to be \emph{used} in a derivation of an irrelevant $\f\rightarrow\f$.

As mentioned in the introduction, relevant logics have usually been studied either as sets of formulas rather than as Tarskian consequence relations, or the Tarskian consequence relation has not been understood to really express the relation of relevant entailment, which is rather expressed by an implication connective~$\rightarrow$. These amount to a similar point, which is that relevant logics have usually been studied in terms of \emph{internal} consequence relations, rather than external ones: that is, they are studied in terms of that relation which obtains between a collection of formulas $\Gamma$ and a formula $\f$ when a formula $\gamma\rightarrow\f$ is valid, where $\gamma$ is a complex formula combining the elements of $\Gamma$ with some kind of conjunction, usually either the lattice conjunction $\land$ or an intensional conjunction $\circ$, commonly called ``fusion''. Both of these choices bring along certain difficulties, and if we use the latter, $\circ$, then, since there are reasons of relevance for rejecting $\f\circ\f\rightarrow\f$, see \cite[\S 29.5]{AB}, we cannot take the internal consequence relation to hold between \emph{sets} of formulas and formulas. In particular, we may not always \emph{add} new copies of formulas and expect to preserve validity. This suggests that an approach employing \emph{multisets}, like that taken in \cite{Avron, MMR, MMR2}, is better suited to the task at hand.\footnote{It should be noted that already in his dissertation, Meyer noted that set consequence was inadequate for relevance purposes for precisely the kind of reasons we have mentioned: see \cite[p. 113]{Meyerdiss}.}

The standard relevant logics, like \textbf{R} and its neighbours, are somewhat complex -- for instance, they do not have finite characteristic matrix presentations -- so let's consider a logic with a simpler presentation which nicely exemplifies some of these difficulties: \emph{Abelian logic}. 

\begin{example} \label{abelset}
We present Abelian logic in the language consisting of binary connectives~$\wedge$, $\vee$, and~$\circ$, a unary connective~$\neg$, and  truth constants~$\0$ and $\1$. Its semantics is given by the algebra $\alg{Z} = \langle \mathbb{Z}, \min,\max,  +, -, 0, 1 \rangle$, which is in fact a characteristic algebra for Abelian logic. We also consider a defined binary connective $\to$, where $\f\to\p$ is defined as $\neg(\f\circ \neg \p)$, and constants for the integers, namely, \Numeral{n+1} defined as $\numeral n\circ\1$ and $\Numeral{-n}$ as $\neg\numeral n$, for all $n\ge 1$.\footnote{For the peculiar properties of Abelian logic (e.g., that intensional conjunction and disjunction coincide) the reader can check \cite{MOG} where one can also find other variants of Abelian logic.}

If we take the usual order $\leq$ on $\alg Z$, then there are three natural things we might try to define a consequence relation; for any formulas $\vectn\f,\p$, let us define:\footnote{We use $\Rightarrow$ as a metalanguage, classical implication and abuse $\f_i,\p$ to mean the evaluation of the formulas in \alg Z on the right-hand side of the symbol~$\models$.}
\begin{alignat*}{3}
\vectn\f &\vdash_\alg{Z}^{p} \p  &\qquad\text{iff} \qquad 	 \alg{Z}&\models 0\leq \min(\vectn\f) \Longrightarrow 0\leq \p
\\ 
\vectn\f &\vdash_\alg{Z}^{\leq} \p  &\qquad\text{iff} \qquad 	 \alg  Z&\models\min(\vectn\f) \leq \p
\\ 
\vectn\f &\vdash' \p  &\qquad\text{iff} \qquad 	  \alg  Z&\models \f_1 + \f_2 + \dots + \f_n \leq \p
\end{alignat*}

\noindent
Let's consider these in turn. The relation $\vdash^{p}_{\alg{Z}}$ is the most natural \emph{external} consequence relation, as we are concerned with preservation of \emph{designated} values (namely, the values $\geq 0$). Indeed, this is a (finitary) Tarskian consequence relation (in particular, it's obviously monotone). However, it does not satisfy the deduction theorem: note that $0\leq \min(1,2)$ implies $0\leq 1$, and yet $0\leq 1$ and $0\nleq 2\rightarrow 1$. A natural way to try to buy this back is to internalise just a bit more, as in $\vdash^\leq_{\alg{Z}}$. This is, again, a (finitary) Tarskian consequence relation, but it also fails to satisfy the deduction theorem (the same choice of values as before suffices to show this). In addition, this also fails to satisfy the deduction rule of modus ponens, as $\min(1\to-2,1)\nleq -2$.

In light of this, the most obvious way to go \emph{more} internal is to use $+$ rather than $\min$ (i.e., $\circ$ rather than~$\land$) to bunch premises, resulting in $\vdash'$. Unfortunately, this definition, which is apparently the intended one, fails to be \emph{well-defined} for sets of premises. By the definition, we have that $1\leq 1$ iff $ \{\overline{1}\}\vdash'\overline{1}$ iff $ \{\overline{1},\overline{1}\}\vdash'\overline{1}$ iff $2\leq 1$. By using \emph{sets} of premises, we have clearly distorted the intended meaning, and produced an obviously bad result here: the definition we want calls out to be given in a multiset framework.



\end{example}

These issues, caused by disregarding the number of occurrences of a formula, motivate a move to multiset consequence.\footnote{For the sake of simplicity, we shall just consider multisets here, though we note that working with sequences or even structures (in the sense of \cite{restall2000}) is also a natural choice. Furthermore, as we'll see in the second half of Section \ref{sect:BCI-Example} and Proposition~\ref{p:onT},
a move in the direction of even more discerning data types is suggested by some limitations of our proposal.} Let us recall now some basic notions.

\begin{definition}
\label{def:multiset}
A \emph{multiset} over a set~$A$ is a mapping $M\colon A\to\mathbb N$.
The natural number $M(a)$, for $a\in A$, is called the \emph{multiplicity} of the element~$a$ in the multiset~$M$.
The set $|M|=\class{a\in A}{M(a)>0}$ is called the \emph{support} of~$M$.

We say that a multiset is \emph{finite} if it has a finite support.\footnote{As mentioned earlier, we'll be concerned throughout with finitary consequence, so our concern will, for the most part, be with finite multisets.}
We use the usual square bracket notation $M=[\vectn a]$ for finite multisets, where $M(a)$ is the number of occurrences of~$a$ in the list \vectn a, provided the domain~$A\supseteq\fclass{\vectn a}$ of~$M$ is clear from the context or does not need to be specified.

We denote the set of all multisets over~$A$ by $\MPow(A)$ and the set of all finite multisets over~$A$ by $\FMPow(A)$.

A set~$B\subseteq A$ is identified with the multiset~$M_B\colon A\to\mathbb N$ such that
$M_B(a)=1$ if $a\in B$ and $M_B(a)=0$ if $a\notin B$
(i.e., $M_B$ is the characteristic function of~$B$ on~$A$). The \emph{empty multiset} will be written $\emptymultiset$.

A multiset~$M\in\MPow(A)$ is a \emph{submultiset} of a multiset~$N\in\MPow(A)$, written $M\le N$, if $M(a)\le N(a)$ for all $a\in A$.
The pointwise operations $\cup,\cap,\uplus,\Msetminus$ on multisets over~$A$ are defined as follows, for all $a\in A$:
\begin{align*}
    (M_1\cap M_2)(a)&= \min\{M_1(a), M_2(a)\}
    &&   \text{\emph{intersection}}
\\
    (M_1\cup M_2)(a)&= \max\{M_1(a), M_2(a)\}
    &&   \text{\emph{union}}
\\
    (M_1\uplus M_2)(a)&=M_1(a)+M_2(a)
    &&   \text{\emph{sum}}
\\
    (M_1\Msetminus M_2)(a)&= \max\{0, M_1(a)-M_2(a)\}
    &&   \text{\emph{difference}}
\end{align*}
\end{definition}

\noindent
Returning to the example above, we can now adapt the definition of $\vdash'$, using multisets, to be what we really want, namely:
\begin{align*}
[\vectn\f] \vdash_\alg{Z} \p  \qquad\text{iff} \qquad 	&  \alg  Z\models \f_1 + \f_2 + \dots + \f_n \leq \p
\end{align*}

This definition does get us what we want -- we have the deduction theorem and the validity of modus ponens, for instance, as well as a variety of other nice properties. The move to multisets, and the ensuing rejection even of the version of monotonicity which we found to fail when premises are combined with~$\circ$ (namely that $[\f,\f]\nvdash\f$), is quite natural. As mentioned above, the rejection of even this weak form of monotonicity is also motivated for reasons of relevance. This, along with the example, motivates our move to multiset consequence as providing a good basis for relevant consequence: essentially, the reason is that when we reject monotonicity, we want to \emph{really} reject monotonicity. Giving such an account will be the goal of the paper. 

\section{From Relevant Tree Proofs to Relevant Consequence Relations}

Our goal, as mentioned, is to obtain the right abstract account of relevant consequence. Our approach will be to start from a concrete approach, defining what it takes for a tree proof to be relevant (in accordance with the use criterion), and working from there. 

The kind of proof system we'll be interested in are axiomatic, composed out of \emph{consecutions,} at first as above, objects of the form $\Gamma\rhd\f$, where $\Gamma$ is a multiset of formulas and $\f$ a formula, and then, later, as objects $\Gamma\rhd\Delta$ where both $\Gamma,\Delta$ are multisets of formulas. In general, a consecution system will comprise a set of axioms and inference rules, as before. After presenting proofs in such a system, we'll present an abstract definition, and relate the latter to the former by way of justification.

We start with a pair of definitions giving us the grist for our mill:

\begin{definition}\label{d:rules}
An \emph{$\form$-consecution} is a pair $\tuple{\Gamma, \f}$, where $\Gamma$ is a finite multiset of formulas from $\form$,
and $\f \in \form$. We call the elements of $\Gamma$ the \emph{premises} and $\f$ the \emph{conclusion} of the consecution. Instead of `$\langle  {\Gamma,\f}\rangle$', we write `$\Gamma  \rhd\f$' to refer to the associated sequent. Slightly abusing notation, we identify the formula $\varphi$ with the consecution $\emptymultiset \rhd \varphi$. 
\end{definition}

\begin{definition}[Axiomatic system]\label{d:AxSys}
An \emph{axiomatic system} in $\form$ is a set of $\form$-consecutions. The elements of an axiomatic system of the form $\Gamma \rhd\f$ are called \emph{axioms} if\/ $\Gamma = \emptymultiset $ and \emph{rules of inference} otherwise. 
\end{definition}

Let us fix the central notion of this section:
tree-proofs in a given axiomatic system~\AS\ from \emph{multisets} of premises (once we've done so, we'll pick out the relevant ones). As in the case of proof from \emph{sets} of premises, proofs will be certain finite trees labeled by formulas. There are two main differences: 
\begin{itemize}
\item As the rules now have multisets of premises, when applying them to obtain formulas labeling non-leaves we have to speak about the multisets of labels of the immediate predecessors.
\item As some formula  could label multiple leaves, we have to make sure that it is among the premises sufficiently many times (unless it is an axiom). 
\end{itemize}

While the former difference requires only minimal changes in the notion of tree proof, the latter one requires a bigger change: to stress its importance we shall formulate it as an extra (fourth) condition in the definition of the tree-proof, even though it makes the second condition redundant.  
Recall that our aim is to capture the notion of \emph{relevant} proof, which, given the use criterion, means that a formula should occur among the premises \emph{exactly} as many times as we need it for the proof. Let's say a bit more about why we take this rather strong reading of the use criterion, requiring that we list premises \emph{as many times} as they're used, in addition to requiring that they are used. As noted in the introduction, in the standard relevant logics, we generally do not have that $(\f\circ\f)\rightarrow\f$ is valid. Indeed, adding this to \textbf{R} gives rise to some rather irrelevant looking consequences, such as $(\f\to\f)\rightarrow\neg(\p\to\p)$.\footnote{For discussion of related points, see \cite{Dunn21, Tedder22}.} This indicates that keeping track not just of what premises are used, but also \emph{how many times} they're used is important for relevance purposes. 

Given our robust version of the use criterion, there remains a subtle problem of how to deal with axioms which could, in principle, occur among premises; our definition below gives, in our opinion, the appropriate condition, but note that other options are available, and some will be discussed in Remark~\ref{rem:relevant-tree-Use}. Roughly speaking, we shall require axioms to be used in the proof \emph{at least} as many times as they appear among the premises (i.e., if you mention it, you have to use it). 

Let us state the definitions and then discuss a bit further.
The notion of tree-proof from the following definition has already been introduced in \cite[Def.~5.26]{cintula2019}; here we additionally give a condition when a tree-proof is relevant (according to the use criterion).

\begin{definition}[Tree-proof]\label{d:tree-proof-multisets}
Let $\AS$ be an axiomatic system. A proof of a formula $\f$ from a multiset of premises $\Gamma$ in $\AS$ is a finite tree $T$ labeled by formulas such that:
\begin{enumerate}
\item Its root is labelled by $\f$
\item A formula labeling a leaf is an axiom or an element of the \emph{support} of $\Gamma$.
\item For each non-leaf node $\node n$, there is a rule of inference $\Delta \rhd \p$ in $\AS$, where $\f$ is the label of $\node n$ and $\Delta$ is the multiset of formulas labeling the immediate predecessors of $\node n$.
\item Each formula $\x$ which is not an axiom of $\AS$ labels at most $\Gamma(\x)$ leaves.
\end{enumerate}

\noindent
We say that $T$ is a \emph{relevant} tree-proof in $\AS$ from a multiset of premises $\Gamma$ if, in addition: 

\begin{itemize}
\item[(R)] Each formula $\x$ labels at least $\Gamma(\x)$ leaves in $T$.
\end{itemize} 

\noindent
If there is a (relevant) \emph{tree-proof} of a formula $\varphi$ in $\AS$ from a multiset of premises $\Gamma$, we write $\Gamma\vdash_{\!\AS} \f$ (resp.\ $\Gamma\vdash_{\!\AS}^r \f$).\footnote{This notion of tree-proof  could be linearized along the lines of the approach in \cite{BaWe16}, but we will leave the details of that to the reader.}
\end{definition}

We use the following notation to simplify the restriction in (relevant) proofs mentioned above: Given a tree $T$ labeled by formulas,  by $\Lambda_T$ we denote the \emph{multiset} of labels of its leaves. Using this notation  we can express the fourth condition in any of the following equivalent ways: 
\begin{itemize}
    \item If $\x$ is not an axiom of $\AS$, then $\Lambda_T(\x) \le \Gamma(\x)$ .
    \item Each element of $|\Lambda_T\Msetminus \Gamma|$ is an axiom of $\AS$.
    \item There is a finite multiset of axioms $\Delta$ such that $\Lambda_T \le \Gamma\uplus\Delta$.
\end{itemize}
The relevant condition can then be simply expressed as $\Gamma \le \Lambda_T$ and its conjunction with the fourth condition is then equivalent to:
\begin{itemize}
    \item There is a finite multiset of axioms $\Delta$ such that $\Lambda_T = \Gamma\uplus\Delta$.
\end{itemize}





\begin{remark}\label{rem:relevant-tree-Use}
Condition (R) is a simple implementation of the Use Criterion, requiring just that in a relevant tree derivation, each premise is used. The way the premises must be used, in particular, is that they should be inputs to one of the rules of inference used to obtain the conclusion. In many cases, there is only one rule of inference, modus ponens, so in such systems, to quote Meyer, \cite[p.~54]{Meyer74}, ``use in a deduction is use in an application of \emph{modus ponens}.'' 

\pagebreak

The definition of relevant proof is further complicated by the fact that some axioms of $\AS$ could appear in $\Gamma$, and there is a question of whether we should require these to be \emph{used} in the same way that non-axioms are. Consideration of this question naturally gives rise to two related alternative definitions we could use to replace the relevance condition. A tree-proof\/~$T$ in $\AS$ from a mutliset of premises $\Gamma$ is: 
\begin{itemize}
\itemsep=0em
\item[(wr)] \emph{weakly relevant} if each element of $|\Gamma \Msetminus \Lambda_T|$ is an axiom, i.e.,  for each $\x$ \emph{which is not an axiom of \AS} we have $\Lambda_T(\x) = \Gamma(\x)$. 
\item[(sr)] \emph{strongly relevant} if $\Gamma = \Lambda_T$. 
\end{itemize}
Using the suggestive notation $\vdash_{\!\AS}^\mathit{wr}$ and\/ $\vdash_{\!\AS}^\mathit{sr}$, we observe that, as the names suggest:
\[
    {\vdash_{\!\AS}^\mathit{sr}} \subseteq {\vdash_{\!\AS}^\mathit{r}}  \subseteq {\vdash_{\!\AS}^\mathit{wr}} \subseteq {\vdash_{\!\AS}} 
\]
The difference is easily demonstrated: consider a set $\form= \{x,y,z\}$ and an axiomatic system $\AS$ with the axioms~$x$ and~$y$ and no rules.
Then the only tree-proof of $x$ in $\AS$ we can write is the one-node tree $T$ whose only node is labeled by $x$, i.e., $\Lambda_T = [x]$. We can easily determine when $T$ is a (weakly/strongly) (relevant) proof of $x$ from a multiset $\Gamma$: 
\begin{multicols}{2}
\begin{itemize}
\itemsep=0em
\item $\Gamma\vdash_{\!\AS} x$ for any multiset $\Gamma$
\item $\Gamma\vdash_{\!\AS}^\mathit{r}x$  iff $\/ \Gamma \le [x]$;  
\item $\Gamma\vdash^\mathit{sr}_{\!\AS} x$ iff $\/ \Gamma = [x]$;
\item $\Gamma\vdash^\mathit{wr}_{\!\AS} x$  iff $\/  \Gamma(z) = 0$; 
\end{itemize}
\end{multicols}

\noindent
Therefore we have:
\begin{itemize}
\itemsep=0em
\item $[x,z]\vdash_{\!\AS} x$, but not $[x,z]\vdash^\mathit{wr}_{\!\AS} x$.
\item $[x,y]\vdash^\mathit{wr}_{\!\AS} x$, but not $[x,y]\vdash^\mathit{r}_{\!\AS} x$.
\item $\emptymultiset\vdash^\mathit{r}_{\!\AS} x$, but not $\emptymultiset\vdash^\mathit{sr}_{\!\AS} x$.
\end{itemize}

\noindent
Note that if \AS\/ has no axioms, then:
$$
{\vdash^{\mathit{r}}_{\!\AS}} = {\vdash^{\mathit{sr}}_{\!\AS}} = {\vdash_{\!\AS}^{\mathit{wr}}}.
$$ 

Throughout, we'll adopt the relevant condition, leaving consideration of weak and strong relevance, but we note the distinction here in order to draw attention to the fact that there is a choice point and to give some indication of why we have gone the way we did.

\end{remark}

\begin{remark}
It's worth noting that the relations $\vdash^r_{\AS}$ and $\vdash_{\AS}$ may coincide,  e.g., whenever $\AS$ contains the inference rules of \emph{Modus Ponens} ($\f,\f\to\p \rhd \p$) and weakening ($\f \rhd \p\to \f$). 

In such a case we can take any tree-proof $T$ of $\f$ from $\Gamma$ and for any formula $\x$  construct a tree-proof $T^\x$ of $\f$ from $\Gamma$  such that  $\Lambda_{T'} = \Lambda_{T}\uplus[\x]$: indeed
it suffices to add a new root labeled $\x\to\f$ connected to the original one by weakening and then adding a new leaf labeled $\x$ and new root labeled by $\f$ again and connecting them using \emph{Modus Ponens}. In a picture, with weakening we can move from the left tree proof below to the right one, increasing the number of uses of $\x$ by one: 

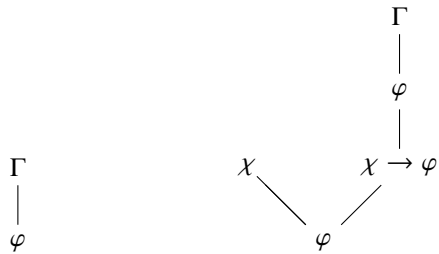
\begin{center}
\begin{tikzpicture}
\node (1) {$\Gamma$};
\node (2) [below of=1]{$\f$};

\draw[-] (1)--(2);
\end{tikzpicture}
\hspace{15mm}
\begin{tikzpicture}
\node (1) {};
\node (2) [below of=1]{};
\node (3) [right of=2]{};
\node (6) [right of=3] {};
\node (4) [left of=6]{$\x$};
\node (5) [below of=6] {$\f$};
\node (7) [right of=6]{$\x\to\f$};
\node (8) [above of=7]{$\f$};
\node (9) [above of=8] {$\Gamma$};
\draw[-] (3)--(5)--(7)--(8)--(9);
\draw[-] (4)--(5);
\end{tikzpicture}
\end{center}

\noindent
Repeating this trick as many times as necessary yields a \emph{relevant} tree-proof $S$ of $\f$ from $\Gamma$. 
\end{remark}

Our next goal is to find some simple description of (relevant) provability relations  ${\vdash_{\!\AS}^\mathit{(r)}}$. Below is our definition appropriate to the work we aim to do in this paper and, as discussed, it differs in key regards from other standard accounts which involve monotonicity and contraction as defining conditions (the latter just in virtue of using sets).


\begin{definition}\label{def:CR}
A \emph{consequence relation} on a set of formulas $\form$ is a relation ${\vdash}$ between finite multisets of formulas and formulas obeying the following conditions for each finite multiset $\Gamma\uplus\Delta\uplus[\f,\p]$ of formulas:\footnote{Throughout we'll state conditions with bare variables as a convention for the universal closure thereof.}
\begin{itemize}
\item  $\f\vdash \f$. \hfill \REFL

\item If\/ $\Gamma,\p\vdash \f$ and $\Delta\vdash \p$, then $\Gamma\uplus\Delta\vdash \f$. \hfill \CUT
\end{itemize}
We add the prefix \emph{monotone} if furthermore: 
\begin{itemize}
\item If\/ $\Gamma\vdash \f$, then $\Gamma\uplus \Delta \vdash \f$. \hfill \MONO
\end{itemize}
We add the prefix \emph{contractive} if furthermore:  
\begin{itemize}
\item If\/ $\Gamma\uplus[\p,\p]\vdash \f$, then $\Gamma\uplus[\p]\vdash \f$. \hfill \CONT
\end{itemize}

\noindent
Furthermore, a consequence relation is \emph{Tarskian} if it is both monotone and contractive. A formula $\x$ such that $\vdash \x$ is called a \emph{theorem} of\/~$\vdash$. 
\end{definition}

We use some usual notational conventions, stated next to their set versions for easy comparison:
\[
\begin{array}{rcl@{\qquad}l}
                        &                       & \text{in sets:}            & \text{in multisets:}\\
\vectn{\f}\vdash\f      & \text{ stands for }   & \{\vectn{\f}\}\vdash\f    & [\vectn{\f}]\vdash\f \\
\Gamma,\Delta\vdash\f   & \text{ stands for }   & \Gamma\cup\Delta\vdash\f  & \Gamma\uplus\Delta\vdash\f\\
\Gamma,\p\vdash\f       & \text{ stands for }   & \Gamma\cup\{\p\}\vdash\f  & \Gamma\uplus[\p]\vdash\f\\
\vdash \f               & \text{ stands for }   & \emptyset \vdash \f       & \emptymultiset\vdash \f
\end{array}
\]

\begin{remark}\label{r:RelToTarski}
It is easy to see that there is a one-one correspondence between finitary Tarskian consequence relations (introduced in Definition~\ref{def:Tarski}) and our monotone contractive consequence relations.
\end{remark}

\begin{example}\label{e:GoodAbel}
Recall that in Example~\ref{abelset} the third of our attempted definitions, the relation $\vdash'$, was actually ill-defined with sets; but the move to multisets avoids the problem. Indeed the relation $\vdash_\alg{Z}$ defined at the end of Section~\ref{sect:Tarskian} is clearly an example of a relation which is non-monotone (e.g., $\vdash \bar0$ but not $\bar1 \vdash \bar0$) and non-contractive (e.g., $\bar1 \vdash \bar1$ but not $\bar1, \bar1 \vdash \bar1$). See the next section for more details and discussion on related logical systems.
\end{example}


\begin{remark}
There's another well-known kind of motivation for non-monotonic consequence, which is discussed at length, e.g., in~\cite{MakinsonD}.\footnote{Here sets are employed, rather than multisets, so we'll follow that convention for the remainder of this remark.} The goal of this kind of non-monotonic consequence is to characterise \emph{defeasible} inference, where a consequence may obtain only in the absence of some \emph{defeaters} -- that is, some premises may have a conclusion as a consequence, but an extension of the premise set by some further premises \emph{defeating} some of the existing premises. A famous example here concerns $\f\coloneqq$ `Tweety is a penguin', $\p\coloneqq$ `Tweety is a bird', and $\x\coloneqq$ `Tweety flies'. The inference from $\p$ to $\x$ is defeasibly correct, but that from $\{\f,\p\}$ to $\x$ is clearly incorrect. So \MONO\ is clearly inappropriate for an account of defeasible inference, but relevant consequence, in our sense, also does not provide an especially good account of this. For instance, note that while $\f$ defeasibly implies $\p$, and $\p$ defeasibly implies $\x$, it is not the case that $\f$ defeasibly implies $\x$, and this is just an instance of \CUT.

Makinson \cite[p.~5]{MakinsonD}, to avoid this sort of problem, considers non-monotonic consequence relations more appropriate to defeasible inference, given by the following two conditions:
\begin{enumerate}
\item[$\bullet$] $\Gamma \vdash \f$, for every $\f \in \Gamma$. \hfill \GENREFL
\item[$\bullet$] If\/ $\Gamma, \Delta\vdash \f$ and $\Gamma \vdash \x$ for every $\x\in \Delta$, then $\Gamma\vdash \f$. \hfill \CUMCUT
\end{enumerate}
\CUMCUT\ clearly follows from \CUT.
On the other hand, \GENREFL\ is clearly appropriate for \emph{monotone} consequence relations only, as we will see in the next lemma. 
One could consider a common generalization of both notions, but doing so would lead us outside the scope of this paper.
\end{remark}

With that remark out of the way, let us note that the following facts are immediate: 

\begin{lemma}\label{l:OnCuts2}
Let ${\vdash}$ be a consequence relation. Then for each multiset of theorems $\Delta$, $n\geq 1$, formulas $\f, \vectn{\x}$, and finite multisets $\vectn{\Delta}$ of formulas: 
\begin{itemize}
\item If\/ $\vectn{\x}\vdash \f$ and $\Delta_i\vdash \x_i$  for each $i\leq n$, then $\vectn{\Delta} \vdash \f$. \hfill \RELCUT
\item If\/ $\Gamma_1,\Delta\vdash \f$, then $\Gamma_1\vdash \f$. \hfill \THMREM
\end{itemize}

\noindent
Furthermore, ${\vdash}$ is monotone iff for each multiset $\Gamma$ of formulas: 
\begin{itemize}
\item  $\Gamma,\f \vdash \f$ \hfill \GENREFL
\end{itemize}
\end{lemma}


Let us note that any consequence relation can be seen as a set of consecutions. This allows us to state and prove several interesting facts. The first one is obvious.

\begin{lemma}
Let $\mathcal X$ be systems of (monotone and/or contractive) consequence relations. Then so is $\bigcap \mathcal X$.
\end{lemma}

Next we prove that the relation $\vdash^{r}_{\!\AS}$ is indeed an example of a consequence relation, and furthermore that it is an especially simple example of such. This goes part of the way towards justifying our abstract definition in terms of the concrete examples which obey the Use Criterion. 

\begin{proposition}\label{prop:least-conseq}
Let $\AS$ be an axiomatic system in $\form$. Then $\vdash^{r}_{\!\AS}$ is the least consequence relation on $\form$ containing $\AS$ and\/ $\vdash^{}_{\!\AS}$ is the least monotone consequence relation on $\form$ containing $\AS$.  
\end{proposition}

\begin{proof}
First we show that $\vdash^{r}_{\!\AS}$ (resp.\ $\vdash_{\!\AS}$) is (monotone) consequence relation on $A$. The \REFL\ for both $\vdash_{\!\AS}$ and $\vdash^{r}_{\!\AS}$ is obvious as a tree with a single node labeled by $\f$ is a (relevant) tree-proof of $\f$ from the premises $[\f]$. The \MONO\ for $\vdash_{\!\AS}$ is also straightforward (clearly any tree-proof $T$ from premises $\Gamma$ is tree-proof from $\Gamma\uplus\Delta$). Finally, we have to deal with \CUT: suppose that we have a (relevant) tree-proof $T$ of $\f$ from $\Gamma \uplus[\p]$ and a (relevant) tree-proof $S$ of $\p$ from $\Delta$ and we have to find a (relevant) proof $R$ of $\f$ from $\Gamma\uplus\Delta$. First note that in the relevant case $\p$ has to label a leaf $\node{l}$ of $T$ (as $\Gamma\uplus[\p] \le \Lambda_T$) and in the \emph{monotonic} we can, without loss of generality,  assume it at well (indeed otherwise already $T$ would be the proof of $\f$ from $\Gamma\uplus\Delta$). Thus we can define $R$ as the tree resulting from $T$ by replacing $\node l$ by the tree $S$ and note that for each formula $\x$ we have:
 $$
 \Lambda_{R}(\x) =   \begin{cases}
  \Lambda_{T}(\x) + \Lambda_{S}(\x)         & \text{if } \x \neq \p\\
   \Lambda_{T}(\x) + \Lambda_{S}(\x) - 1   & \text{if } \x = \p
 \end{cases}
$$
To show that $R$ is indeed a tree-proof of $\f$ in $\AS$ from $\Gamma\uplus\Delta$ we consider a formula $\x$ which is not an axiom and recall from the assumptions we know that:
 $$
 \Lambda_{T}(\x) \le (\Gamma\uplus[\p])(\x)  \qquad\text{and}\qquad   \Lambda_{S}(\x) \le \Delta(\x)
 $$
 Thus for $\x \neq \p$ we get:
  $$
 \Lambda_{R}(\x) =  \Lambda_{T}(\x) + \Lambda_{S}(\x) \le (\Gamma\uplus[\p])(\x) + \Delta(\x) = \Gamma(\x) + \Delta(\x) = (\Gamma\uplus \Delta)(\x)
 $$
and for  $\x = \p$ we get:
$$
 \Lambda_{R}(\p) =  \Lambda_{T}(\p) + \Lambda_{S}(\p) - 1 \le (\Gamma\uplus[\p])(\p) + \Delta(\p) - 1 = \Gamma(\p) + 1  + \Delta(\p) - 1 = (\Gamma\uplus \Delta)(\p) 
$$
The proof that $R$ is relevant if $T$ and $S$ are is analogous; we only observe that in this case we have inequalities converse to those above even for the axioms.

The fact that $\AS \subseteq{\vdash^{(r)}_{\!\AS}}$ is obvious. In order to show the minimality condition, let us focus on the relevant case (the monotonic one is just its simpler variant). Consider any consequence relation $\vdash$ such that $\AS\subseteq{\vdash}$; we need to show that ${\vdash^{r}_{\!\AS}}\subseteq\,\vdash$.  Assume that we have a relevant tree-proof $T$ of $\f$ from $\Gamma$; for each node~$\node n$, let $\f_\node n$ denote the formula labeling it and let $\Gamma_\node n$ denote the multiset of formulas labeling the leaves of the subtree of $T$ containing the node $\node n$ and all its predecessors. 

We prove by induction over the tree that $\Gamma_\node n\vdash \f_\node n$ (in the monotone case we prove directly that $\Gamma\vdash \f_\node n)$. For $\node n$ being a leaf we have $\Gamma_\node n = [\f_\node n]$ and so the claim follows
by the \REFL\ of $\vdash$. Let $[\node n_1, \dots, \node n_k]$ be the multiset of immediate  predecessors of $\node n$; then we know that $\f_{\node n_1}, \dots, \f_{\node n_k}\rhd \f_\node n \in \AS$ and thus also $\f_{\node n_1}, \dots, \f_{\node n_k}\vdash \f_\node n$ (because $\AS \subseteq{\vdash}$). By the induction hypothesis we have $\Gamma_{\node n_i}\vdash \f_{\node n_i}$ for each $i \leq k$ and thus by the \RELCUT\ we obtain  $\Gamma_{\node n_1}, \dots, \Gamma_{\node n_k} \vdash \f_\node n$. Observe that $\Gamma_\node n = \Gamma_{\node n_1} \uplus \Gamma_{\node n_2} \dots  \uplus \Gamma_{\node n_k}$, which completes the proof of this part.

To complete the whole proof it suffices to observe that for the root $\node r$ we have $\Gamma_\node r = \Lambda_T$ thus $\Lambda_T \vdash \f$. Because $T$ is relevant proof we know that  there is a finite multisetset $\Delta$ of axioms of $\AS$ such that $\Lambda_T = \Gamma \uplus \Delta$. From the fact that $\AS \subseteq{\vdash}$ we know that $\Delta$ consists of theorems of $\vdash$, and therefore by \THMREM\ we know that  $\Gamma\vdash \f$.
\end{proof}

So we can obtain the least (relevant) consequence relation for $\AS$ by taking that generated by its (relevant) tree proofs as defined above, and these were defined so as to guarantee observation of the Use Criterion. So while the abstract consequence relations we have defined simply remove monotonicity and a form of contraction from the Tarskian case, nonetheless this relatively minor adaptation is adequate to capture the behaviour of the concrete cases in which we've been most interested. Indeed, we can say a bit more to justify this connection.


\begin{definition}
Let $\vdash$ be a consequence relation. We say that an axiomatic systems $\AS$ is a (relevant) presentation of $\vdash$ whenever ${\vdash} = {\vdash^{(r)}_{\!\AS}}$. 
\end{definition}

Clearly any consequence relation $\vdash$ can be seen as an axiomatic system and so due to Proposition \ref{prop:least-conseq}
it is its own relevant presentation (i.e., ${\vdash} = {\vdash^r_{\vdash}}$); if furthermore $\vdash$ is monotone, then it is its own presentation (${\vdash} = {\vdash_{\vdash}}$). Therefore we can obtain a variant of \L{}o\'s--Suszko theorem.

\begin{theorem}[\L{}o\'s--Suszko]\label{szko}
Every consequence relation $\vdash$ on $\form$ has a relevant presentation. Every monotone consequence relation $\vdash$ on $\form$ has a presentation.
\end{theorem}

It is worth noting that the previous theorem, perhaps surprisingly, implies that every \emph{monotone} consequence relation has a \emph{relevant} presentation. However we have seen already in Remark~\ref{r:TarskiCanBeRelevant}, that we could have ${\vdash_{\!\AS}} = {\vdash^{r}_{\!\AS}}$ which implies ${\vdash^{r}_{\!\AS}}$ could be a monotone consequence relation: in general, relevance of proof is a property that holds regardless of one's choice of axioms and inference rules. If one chooses irrelevant axioms and rules, then making the proofs `relevant' in our sense won't do you much good as far as avoiding relevance is concerned.\footnote{A related point was made by Bennett \cite{Bennett65} in response to Anderson and Belnap: he noted that the Use Criterion by itself does not ensure relevance as you can just add further axioms and rules to obtain all the irrelevant consequences you like. It is, thus, perhaps not surprising that our approach, reliant on the Use Criterion as it is, has the same property.} The next two propositions (with obvious proofs) explicate this phenomenon. 

\begin{proposition}
Let $\AS$ be an axiomatic system. Then $\vdash^{r}_{\!\AS}$ is a monotone consequence relation iff 
$$
{\vdash_{\!\AS}} = {\vdash^{r}_{\!\AS}}
$$
\end{proposition}

\begin{proposition}
Let $\vdash$ be a consequence relation. Then for the least monotone consequence relation $\vdash_{m}$ containing $\vdash$, known as the \emph{monotonic companion} of\/~$\vdash$, we have:
$$
\Gamma \vdash_{m} \f \qquad\text{iff}\qquad \Delta \vdash \f  \quad\text{for some } \Delta\subseteq \Gamma.
$$ 
Furthermore, any relevant presentation of\/~$\vdash$ is a presentation of\/~$\vdash_m$.
\end{proposition}


\section{An Example: \BCI\ and some related systems}
\label{sect:BCI-Example}

Let's take some time to consider some concrete examples of relevant consequence relations over some familiar logics in the vicinity of the relevant logic family. We've seen a bit of Abelian logic as an example, adding some justification to using multisets, but there are a number of better known systems we can study using these tools. Let's start with the implication fragment of the well-known substructural logic \BCI:

\begin{itemize}
\itemsep=0em
\item[(\textsf{I})] $\f\rightarrow\f$
\item[(\textsf{B})] $(\f\rightarrow\p)\rightarrow((\x\rightarrow\f)\rightarrow(\x\rightarrow\p))$
\item[(\textsf{C})] $(\f\rightarrow(\p\rightarrow\x))\rightarrow(\p\rightarrow(\f\rightarrow\x))$
\item[(mp)] $\f\rightarrow\p,\f\rhd\p$
\end{itemize}

This is a sublogic of very many systems, including the relevant logic \R, about which we'll have more to say shortly, as well as fuzzy logics, intuitionistic logic, and others. It's easy to extend this system by an intensional conjunction (or fusion) connective~$\circ$, with the additional axioms:

\begin{itemize}
\itemsep=0em
    \item[(Res$_1$)] $((\f\circ \p)\rightarrow \x)\rightarrow(\f\rightarrow(\p\rightarrow\x))$
    \item[(Res$_2$)] $(\f\rightarrow(\p\rightarrow \x))\rightarrow((\f\circ\p)\rightarrow\x)$
\end{itemize}

In general, we'll just refer to \BCI, and let context determine which connectives are around (for now, it's just~$\rightarrow$ and~$\circ$). First off, it's an obvious, nice feature of the definition of relevant tree proof that we can obtain a proof of the deduction theorem for \BCI\ and some relevant logics expanding it (as we'll see, there are some problems that come up for considering the standard relevant systems in their full vocabularies, but we can give this argument for at least some of the salient systems).

\begin{proposition}\label{p:DeductionTheorem}
$\Gamma,\f\vdash_{\BCI}^r\p\iff\Gamma\vdash_{\BCI}^r\f\rightarrow\p$.
\end{proposition}

\begin{proof}

The proof is straightforward, and is simplified by the fact that if $\Gamma\vdash_{\BCI}\f$ and $\Delta\vdash_{\BCI}\p$ then $\Gamma,\Delta\vdash_{\BCI}^r\f\circ\p$ (this is immediate given (Res$_1$)). We'll just sketch the proof and leave it to the reader to fill in details. From right to left is trivial, just relying on (mp), so for the converse, we'll proceed by induction on the complexity of proofs, supposing that $\Gamma,\f\vdash_{\BCI}^r\p$. The only rule is (mp), and so (disregarding some simple cases such as $\Gamma = \emptymultiset$ and $\f = \p$) the last step of such a derivation must be an application of this, and so there are some submultisets $\Gamma_1,\Gamma_2$ of $\Gamma\uplus[\f]$ such that one of these entails $\x\rightarrow\p$ and the other $\x$. $\f$ may occur in either $\Gamma_1$ or $\Gamma_2$, but the particular occurrence highlighted in the presupposition must have a unique occurrence in one of them, and so depending on where it occurs, either $\Gamma_1'\vdash_{\BCI}^r\f\rightarrow(\x\rightarrow\p)$ or $\Gamma_2'\vdash_{\BCI}^r\f\rightarrow\x$, where $\Gamma_i'$ is $\Gamma_i$ with one fewer occurrence of $\f$. In one case, we just need to appeal to the fact that $\vdash_{\BCI}((\f\rightarrow\x)\circ(\x\rightarrow\p))\rightarrow(\f\rightarrow\p)$, and for the other case (Res$_2$) will do.
\end{proof}

This first fact indicates that, at least in some cases, getting what we want: that is, that our definition results in a match between the \emph{external} consequence relation (of the kind we've defined) and the \emph{internal} consequence relation, given by provable implication in the logic. \BCI\ is a particularly nice logic, and this result extends to such extensions as the `intensional' fragment of \textbf{R} (including implication and multiplicative conjunction and disjunction). However, there are limitations, related to our use of multisets rather than some more discerning kind of structure. First we'll note that the above result fails in some logics weaker than \BCI, and then go on to discuss issues with incorporating the lattice conjunction of \textbf{R}, which is usually (for reasons we'll mention) formulated as a meta-rule.  

For the first point, note about the inclusion of the permutation axiom (\textsf{C}) that it ensures that $(\p\circ(\f\circ\x))\rightarrow((\p\circ\f)\circ\x)$ is derivable in \BCI. This, along with the prefixing axiom (\textsf{B}), ensures that $\circ$ is associative (hence showing that \BCI\ is an extension of the associative Lambek calculus, with the left division as a notational variant of the right). Note, however, that $\circ$ may not associative in the case in logics without (\textsf{C}), and that this has a substantial consequence for the deduction theorem just proved. Consider the relevant logic \T$_{\rightarrow,\circ}$, the implication/fusion fragment of \T, which can be axiomatised by replacing (\textsf{C}) with the suffixing axiom (\textsf{B}$'$) and the contraction axiom (\textsf{W}):
\begin{itemize}
\itemsep=0em
\item[(\textsf{B}$'$)] $(\f\rightarrow\p)\rightarrow((\p\rightarrow\x)\rightarrow(\f\rightarrow\x))$
    \item[(\textsf{W})] $(\f\rightarrow(\f\rightarrow\p))\rightarrow(\f\rightarrow\p)$ 
\end{itemize}

The failure of the associativity of $\circ$ gives rise, in $\T_{\rightarrow,\circ}$, to a remarkable failure of the deduction theorem, given our definitions, as we can see with a simple matrix argument. We have one direction immediately, given the inclusion of (mp), but the converse direction fails.

\begin{proposition}\label{p:onT}
$\f\rightarrow\p\vdash^{\mathit{r}}_{\T_{\rightarrow,\circ}}(\f\circ\x)\rightarrow(\p\circ\x)$ \; and \; $\nvdash_{\T_{\rightarrow,\circ}}(\f\rightarrow\p)\rightarrow((\f\circ\x)\rightarrow(\p\circ\x))$.
\end{proposition}

\begin{proof}
The positive part is straightforward, so we'll just present the matrix to show the negative part: the desired valuation has $\f\mapsto2,\p\mapsto0,\x\mapsto1$, where $\f$ is true on a valuation $v$ just when $v(\f)=3$.\footnote{This matrix was found using Slaney's \texttt{MaGIC}. For information, and to download the software, see \url{http://users.cecs.anu.edu.au/~jks/magic.html}.}
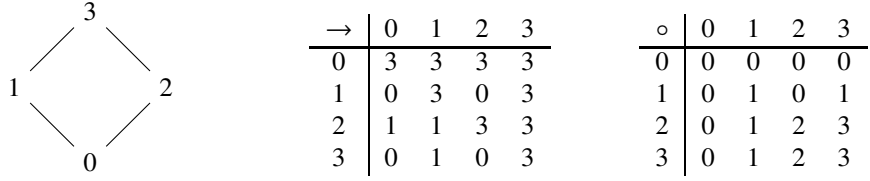
\begin{center}
\begin{tikzpicture}[baseline=8mm]
  \node (0) at (0,0) {0};
  \node (1) at (1,1) {2};
  \node (2) at (-1,1) {1};
  \node (3) at (0,2) {3};
  
  \draw (0) -- (2) -- (3) ;
  \draw (0) -- (1) -- (3);
\end{tikzpicture}
\hspace{15mm}
\begin{tabular}{c|cccccc}
$\rightarrow$ & 0 & 1 & 2 & 3 \\\hline
            0 & 3 & 3 & 3 & 3 \\
            1 & 0 & 3 & 0 & 3 \\
            2 & 1 & 1 & 3 & 3 \\
            3 & 0 & 1 & 0 & 3 \\
\end{tabular}
\hspace{10mm}
\begin{tabular}{c|cccccc}
      $\circ$ & 0 & 1 & 2 & 3 \\\hline
            0 & 0 & 0 & 0 & 0 \\
            1 & 0 & 1 & 0 & 1 \\
            2 & 0 & 1 & 2 & 3 \\
            3 & 0 & 1 & 2 & 3 \\
\end{tabular}
\end{center}
\end{proof}

\noindent
This fact also holds for \E$_{\rightarrow,\circ}$\/, which retains a (very) weak form of permutation but, like \T$_{\rightarrow,\circ}$,  doesn't validate $\f\circ(\p\circ\x)\rightarrow(\f\circ\p)\circ\x$ and, hence, is not \emph{fully} associative. However, the famous relevant logic $\textbf{R}_{\rightarrow,\circ}$, which extends \BCI\ by (\textsf{W}) does enjoy the deduction theorem (and note, \BCI\ is the implication-fusion fragment of the contraction-free relevant system \textbf{RW}). 

Among the usual relevant logics studied in \cite{AB, Routleyetal1982}, \R\; and \textbf{RW} are more or less the upper limit of deductive strength, and most of the other systems do not have (\textsf{C}), so among the usual family of relevant logics, $\vdash^{r}$ only `externalises' the internal consequence relation (that expressed by the provability of a formula of the form $\f_1\circ\dots\circ\f_n\rightarrow\p$) of a handful of (famous, well-studied) systems. For the others, it seems that even the use of multisets brings in too much structure, and that a move to sequences or even trees (as the data types of premises) is necessary. This is a natural avenue for future research, but we leave it aside, focusing on the simple case with multisets.

On to the second point, in order to add a \emph{lattice} conjunction $\land$ to the relevant systems of \cite{AB, Routleyetal1982}, it's not adequate just to use axioms and inference rules, as we've done so far. In particular, these systems have it that if $\f$ and $\p$ are both derivable from no premises, then $\f\land\p$ is so derivable. Importantly, this fact is not reflected in an axiom, such as the usual $\f\rightarrow(\p\rightarrow(\f\land\p))$, as this just delivers $\f\rightarrow(\p\rightarrow\f)$ in that context, the avoidance of which is one of the main aims of the relevant enterprise (and the fact that $(\f\land\p)\rightarrow(\f\land\p)$ is derivable does us no good at all). Similarly, using an inference rule won't help, as then, given \CUT, if we had $\f\,\p\rhd\f\land\p$, we'd have $\f\land\p\vdash^r\f$ and thus $\f,\p\vdash^r\f$, buying us back all the monotonicity we have worked to avoid. What we need to add will have the effect of what Smiley \cite{Smiley63} called a \emph{rule of proof}: something of the form ``if $\vdash\f_1,\dots,\vdash\f_n$ then $\vdash\p$''. This is the form of a simple kind of \emph{metarule}, with zero premises, and perhaps the most famous example is the \emph{necessitation} rule of the local consequence relations of modal logics. The usual rule of \emph{adjunction} in relevant logics is given in the form ``if $\vdash_{\textbf{R}}\f$ and $\vdash_{\textbf{R}}\p$, then $\vdash_{\textbf{R}}\f\land\p$'', and as mentioned this is for a good reason. In some relevant logics much weaker than \textbf{R}, such as those studied in \cite{Routleyetal1982}, we need to add further rules of proof (such as rule versions of prefixing, suffixing, and contraposition).

So, the full relevant systems do not fall under the definition of ``axiom system'' we gave before without some tweaking -- for instance, by taking all the \emph{theorems} of some other presentation of $\R$ and making \emph{that} the set of axioms and using the rule (mp) we could do it (though that smells a little of cheating, it is an option).\footnote{One more system in the vicinity worth mentioning is \R-mingle, or \RM. This can be obtained from \R\; by adding the axiom:
\begin{itemize}
    \item[(\textsf{M})] $\f\rightarrow(\f\rightarrow\f)$
\end{itemize}

\noindent Adding this has the effect of making it the case that $\f,\f\vdash_{\RM_{\rightarrow,\circ}}^r\f$ is valid, in effect allowing us to \emph{multiply} occurrences of premises; combined with (\textsf{W}), allowing us to prove $\f\vdash^r\f\circ\f$, and so allowing us to contract occurrences of premises, this `undoes' the multiset effect, leaving us with, in effect, sets as the data type. This is part of a wider point, that particular axioms and rules will have the effect of collapsing distinctions our setup has been fine-tuned to draw (as we've also seen with adding $\f,\p\rhd\f\land\p$ to systems, and buying back monotonicity that way). This just highlights the point, mentioned earlier, that the use criterion does not pin down systems which are `relevant' in any especially rich sense, but rather just tells you when a derivation in a system is `relevant', even if the system is itself irrelevant in some other senses.} 
The point of this discussion of limitations is to suggest further avenues for improving on the proposal we give here, and motivates the claim that doing so is desirable to capture a wider range of relevant logics (with all their connectives). This task we leave for another occasion, focusing on the simpler setting of multiset-consquence. 

\section{Symmetric Relevant Consequence Relations}


With our adaptation of Tarskian consequence settled, a natural next step is toward multiple conclusion (or ``Scott'') consequence relations. As mentioned, the approach we'll take here is inspired by Bolzano, employing a `conjunctive' reading of the conclusion.\footnote{Note that this feature of our approach causes it to differ from another approach to multiple conclusion in a relevant setting, taking the premises to be combined in accordance with multiplicative conjunction $\circ$ and the conclusions in accordance with multiplicative disjunction $+$, as in \cite[p. 221]{DunnH}. One of the issues with this approach in general is that while in some relevant logics, such as \textbf{R}, $+$ is definable, in other systems where no such definition is possible there is no general consensus on what the properties of $+$ ought to be (though see \cite{Read} for a proof theoretic approach).} This is simple, and in keeping with a common approach (for instance in \cite{Font}), so we'll adopt it, leaving investigations into other options for future work. 

The rough picture we're trying to capture is that a collection of premises should have a collection of conclusions as consequences just in case each conclusion is a consequence of some of the premises (enforcing the relevance constraints we have taken on board in the single conclusion case, so that each premise is used in order to obtain \emph{some} conclusion or other). This is useful as a heuristic but it doesn't pin down any particular formal representation as yet. One choice point concerns when to take a multiset of premises to have the \emph{empty} (multi)set of conclusions as a consequence. In the disjunctive reading, familiar from Gentzen systems, an empty conclusion expresses the \emph{falsum}, but this seems inappropriate for the conjunctive reading. Should we take it that only the empty (multi)set of formulas has itself as a consequence? This approach has the advantage of simplicity, but it is quite out of step with the usual, monotone, case. In that case, if a set of premises has the empty set as a consequence, then it also has every theorem as a consequence -- this follows immediately from the usual definition. Without monotonicity, however, this is not a consequence of the definition, unless we build it in explicitly. To a certain extent, the choice here is arbitrary, but so far we have worked with consequence relation notions that collapse into the Tarskian versions when we reimpose the Tarskian conditions: taking this on board as a desideratum suggests that the treatment of empty conclusions should be similar.  

Let's get to the definitions, starting from a relatively simple case. If $\vdash$ is a monotone and contractive consequence relation, we can define a symmetrization $\vdash^s$ in a simple way, letting $\Gamma\uplus[\vectn{\chi}]$ be a finite multiset of formulas:
$$
\Gamma \vdash^s [\vectn{\chi}]  \qquad\text{iff} \qquad \Gamma \vdash \x_i \text{ for each }\ i\leq n \\ 
$$
Note that this definition has ``right-side''
contraction built in; indeed we obviously have:\footnote{We'll adopt analogous notational conventions erasing $\uplus$ symbols and square brackets from multisets on the right, we will however keep square brackets whenever the right-hand side is a finite (potentially one-element or even empty) multiset given by list (in order to stress the multi-conclusion nature of the involved symmetric consequence relation).}
\begin{itemize}
\item If\/ $\Gamma\vdash^s \Delta,\f,\f$, then $\Gamma\vdash^s \Delta,\f$. \hfill \text{r-}\CONT
\end{itemize}
(Clearly we do not have
right-side
monotony, just in virtue of the conjunctive reading of the conclusions. This holds even in the standard case.)

Removing contraction but keeping monotony makes things a bit more complicated, but still rather intuitive. The basic idea is to pay attention to what parts of $\Gamma$ are needed to derive each $\chi_i$: so we `partition' $\Gamma$ into the multiset union of a collection $\Gamma_1,\dots,\Gamma_i$ of multisets, requiring that each $\chi_i$ is such that $\Gamma_i\vdash\chi_i$ holds for some $\Gamma_i$. So we don't build in contraction, as we're paying attention to how often each premise among $\Gamma$ is used, even when multiple copies may be used to obtain different $\chi_i$'s among the conclusions. Stated formally:
\begin{align*}
\Gamma \vdash^s [\vectn{\chi}]  \qquad\text{iff} \qquad 	& 
										  \text{for each $i \leq n$ there is } \Gamma_i \le \Gamma  \text{ such that } \Gamma_i \vdash \x_i \text { and } \Gamma_1\uplus\dots\uplus\Gamma_n \le \Gamma 
\end{align*}

\noindent
In full generality  we will obviously want the relation of equality instead of submultisethood to obtain between the set of premises the multiset union of its subsets needed to prove each of the conclusions, note however that such a definition would entail that $\Gamma \vdash^s \emptymultiset$ holds iff $\Gamma = \emptymultiset$. As mentioned earlier, however, this is out of step with the usual, Tarskian, case. So we'll build in the usual behaviour (that some premises $\Gamma$ have $\emptymultiset$ as a consequence just in case $\Gamma$ already has every theorem as a consequence). These considerations give rise to the following definition:
\begin{align*}
\Gamma \vdash^s [\vectn{\chi}]  \qquad\text{iff} \qquad 	& 
										  \text{for each $i \leq n$ there is } \Gamma_i \le \Gamma  \text{ such that } \Gamma_i \vdash \x_i \text { and } \Gamma_1\uplus\dots\uplus\Gamma_n = \Gamma \\
\Gamma \vdash^s \emptymultiset  \qquad\text{iff} \qquad 	&  \Gamma \vdash \x \text{ for each theorem } \x \text{ of } \vdash 
\end{align*}
\noindent
It is easy to see that by adding the monotony (and contraction), this general definition gives rise to the more standard one with which we started. With our intuitions fixed and formally expressed, let us formulate the essential properties of symmetric consequence relations.

\begin{definition}\label{def:SCR}
A \emph{symmetric} consequence relation on a set of formulas $\form$ is a binary relation ${\vdash}$ on finite multisets of formulas  obeying the following conditions for each finite multiset $\Gamma\uplus\Delta\uplus\Psi$ of formulas:
\begin{itemize}
\item  $\Gamma\vdash\Gamma$. \hfill \REFL

\item If\/ $\Gamma\vdash \Delta$ and $\Delta\vdash \Psi$, then $\Gamma\vdash \Psi$. \hfill \TRAN

\item If\/ $\Gamma\vdash \Delta$, then $\Gamma,\Psi\vdash \Delta,\Psi$. \hfill \COMP

\end{itemize}
We add the prefix \emph{monotone} if furthermore: 
\begin{itemize}
\item If\/ $\Gamma\vdash \Delta$, then $\Gamma,\Psi\vdash \Delta$. \hfill \MONO
\end{itemize}
We add the prefix \emph{contractive} if furthermore:  
\begin{itemize}
\item If\/ $\Gamma,\p,\p\vdash \Delta$, then $\Gamma,\p\vdash \Delta$. \hfill \CONT
\item If\/ $\Gamma\vdash \Delta,\p,\p$, then $\Gamma\vdash \Delta,\p$. \hfill \text{r-}\CONT
\end{itemize}
\end{definition}

Let us the define the asymmetric variant $\vdash^a$ of a symmetric consequence relation $\vdash$ as follows: 
$$
\Gamma \vdash^{a} \f \qquad\text{iff}\qquad \Gamma \vdash [\f].
$$ 
\noindent 
With this, we can prove the following proposition, showing that that symmetric and asymmetric variants of consequence relations interact as one might hope.
\begin{proposition}\label{p:OnSymmetrization}\mbox{}
\begin{itemize}
\item[(i)] If\/ $\vdash$ is a (monotone and/or contractive) consequence relation, then $\vdash^s$ is (monotone and/or contractive) symmetric consequence relation and
$$
{\vdash} = {(\vdash^s)^a}
$$
\item[(ii)] If\/ $\vdash$ is a (monotone and/or contractive) symmetric  consequence relation, then the relation $\vdash^a$ is a (monotone and/or contractive) consequence relation  and
$$
{\vdash} \supseteq {(\vdash^a)^s}
$$
\item[(iii)] If\/ $\vdash$ is a monotone and contractive symmetric consequence relation, then 
$$
{\vdash} = {(\vdash^a)^s}
$$
\end{itemize}
\end{proposition}

\begin{proof}
(i): In the proof of the first part (i), only transitivity poses a potential problem: let us assume that $\Gamma \vdash^s [\vectk{\delta}]$ and 
$[\vectk{\delta}] \vdash^s [\vectn{\p}]$ and show that $\Gamma \vdash^s [\vectn{\p}]$.

Let us first deal with the case where $k=0$: if $n = 0$ as well then the claim is trivial. Otherwise from the second assumption we know that the $\p_i$s are theorems of $\vdash$ and so, due to the first assumption, $\Gamma\vdash \p_1$ holds and thus fixing $\Gamma_1 = \Gamma$ and $\Gamma_{i+1} = \emptymultiset$ for each $i < n$, we obtain the desired partition of $\Gamma$, which completes the proof. 

Next assume that $k > 0$ and $n =0$: then from the second assumption we know that $\vectk{\delta}\vdash \x$ holds for each theorem $\x$ of $\vdash$
and, due to the first assumption, there are multisets $\vectk{\Gamma}$ such that $\Gamma_i \vdash \delta_i$, and $\Gamma_1\uplus\dots\uplus\Gamma_k = \Gamma$. Therefore due to \RELCUT\ we know that $\Gamma\vdash \x$ as required. 

The case where $k > 0$ and $n > 0$ is only a bit more complex:  we know that there are multisets $\vectk{\Gamma}$ and  $\vectn{\Delta}$ such that  $\Gamma_1\uplus\dots\uplus\Gamma_k = \Gamma$ and $\Delta_1\uplus\dots\uplus\Delta_n = [\vectk{\delta}]$ and for each $i\leq k$ and $j\leq n$ we have
$\Gamma_i \vdash \delta_i$ and $\Delta_j \vdash \psi_j$. Without loss of generality, we can assume that there is a non-decreasing sequence $0 = s^1 \leq s^2 \leq \dots \leq s^k < k$ of integers and a sequence $n^1, \dots, n^{k}$ of non-negative integers such that $\Delta_j = [\delta_{s^j+1},\dots, \delta_{s^j + n^j}]$ (note that this includes the possibility that $n = 0$, i.e., $\Delta_j = \emptymultiset$). To complete the proof it suffices to set $\Gamma_j = \Gamma_{s^{j}+1} \uplus \dots \uplus \Gamma_{s^j + n^j}$.

The monotone and/or contractive conditions are easily checked and rest of the proof of the first claim is straightforward, so we have that: 
$$
\Gamma \vdash \f \qquad\text{iff}\qquad \Gamma \vdash^s [\f] \qquad\text{iff}\qquad \Gamma \mathrel{(\vdash^s)^{a}} \f
$$ 

(ii): For the first part of the proof we will show, on the assumption that $\vdash$ is a symmetric consequence relation,  that  $\vdash^a$  satisfies \CUT \ and \REFL. The second part is immediate, since $\f \vdash \f$ holds by the hypothesis on $\vdash$. So assume that \/ $\Gamma, \p\vdash^a \f$ and $\Delta\vdash^a \p$, in order to show that $\Gamma, \Delta\vdash^a \f$ or, equivalently, that $\Gamma,\Delta\vdash \f$. We have that $\Gamma,\p\vdash \f$ and $\Delta\vdash \p$, and by \COMP, it follows that $\Gamma,\Delta\vdash \Gamma, \p$ and, hence, one application of \TRAN\ gives that $\Gamma, \Delta\vdash \f$. Checking \MONO\ and \CONT\ are also simple exercises left to the reader.

Finally, we verify that ${\vdash} \supseteq {(\vdash^a)^s}$. By definition,  $\Gamma \mathrel{(\vdash^a)^{s}} [\vectn{\chi}]$ only if 
$\text{for each $i \leq n$ there is } \Gamma_i \le \Gamma  \text{ such that } \Gamma_i \vdash^a \x_i \text { and } \Gamma_1\uplus\dots\uplus\Gamma_n = \Gamma$. The latter, in turn, implies that for each such $\Gamma_i$, we have that $\Gamma_i \vdash \x_i$. Now, given $\Gamma_2 \vdash \x_2$, by \COMP, we have that $\x_1,\Gamma_2 \vdash [\x_1, \x_2]$, and similarly since $\Gamma_1 \vdash [\x_1]$, we may obtain that $\Gamma_1, \Gamma_2 \vdash \x_1, \Gamma_2$   and hence, by \TRAN, we get $\Gamma_1, \Gamma_2 \vdash [\x_1, \x_2]$.   Proceeding in this manner, it follows that  $\Gamma_1\uplus\dots\uplus\Gamma_n = \Gamma \vdash [\vectn{\chi}]$, which is what we wanted.

(iii): In the presence of monotonicity and contraction, our definition of the $\vdash^s$ relation is the one at the beginning of the section. We can see that in this case, $\x_i \vdash \x_i$ by \REFL, and then $[\vectn{\chi}] \vdash [\x_i]$ by \MONO. Hence,  $\Gamma \vdash [\vectn{\chi}]$ implies that $\Gamma \vdash [\x_i]$ by \TRAN, so we have  that for each $i \leq n$, $\Gamma \vdash^a \x_i$  and hence, $\Gamma \mathrel{(\vdash^a)^{s}} [\vectn{\chi}]$ as desired.
\end{proof}


Let us consider two examples of different ways of symmetrizing an asymmetric consequence relation, and ways these options can diverge. In so doing, we'll show that the converse inclusion of point (ii) of the previous proposition can fail.

\begin{example}
We stated point (ii) in the previous proposition as an inclusion rather than an equality, and we can give a simple counterexample to the converse inclusion. Consider a singleton set of formulas $\form = \{x\}$ and a relation $\vdash$ on multisets on $A$ defined $\Gamma \vdash \Delta$ iff $\Gamma = \Delta$ or $\Gamma(x) > \Delta(x) \geq 2$. Clearly 
\begin{itemize}
\item $\vdash$ is symmetric consequence relation (note that $\vdash$ is not monotone as $x,x\nvdash [x]$), 
\item $\Gamma\vdash^a x$ iff $\Gamma = [x]$
\item we do not have $x,x,x \mathrel{(\vdash^a)^s} [x,x]$, and thus $(\vdash^a)^s$ and $\vdash$ are distinct symmetrizations of $\vdash^a$.
\end{itemize}

\end{example}
\noindent
This shows the principle of the thing, but we can give a more natural example.

\begin{example}\label{e:symAbel}
Recall the consequence relation $\vdash_\alg{Z}$, given in Example~\ref{e:GoodAbel}, and define an alternate symmetrization $\vdash$ as follows (recall that we treat the empty sum as $0$)
$$
\f_1, \dots, \f_m \vdash [\vectn\p]  \qquad\text{iff} \qquad 	  \alg  Z \models \f_1 + \f_2 + \dots + \f_m \leq \p_1 + \p_2 + \dots + \p_n
$$
It is easy to see that ${\vdash^a} = {\vdash_\alg{Z}}$ but ${\vdash} \neq {\vdash^s_\alg{Z}}$ and so ${\vdash} \neq (\vdash^a)^s$.
Indeed we have $\vdash[\overline{1},\overline{-1}]$, but since $\nvdash_\alg{Z}\overline{-1}$, it follows that $\nvdash^s_\alg{Z}[\overline{1},\overline{-1}]$.

Another example could be built using the logic BCI expanded by conjunction $\circ$ and its unit, the truth constant $t$ described by axioms $\f\circ t\to\f$ and $\f\to t\circ \f$: let us define an alternate symmetrization $\vdash$ for it as (treating the empty conjunction as constant $t$)
$$
\f_1, \dots, \f_m \vdash [\vectn\p]  \qquad\text{iff} \qquad \vdash_{\BCI} \f_1 \circ \f_2 \circ \dots \circ \f_m \to \p_1 \circ \p_2 \circ \dots \circ \p_n
$$
As in the previous case, it is easy to see that ${\vdash^a} = {\vdash_\BCI}$ (thanks to Proposition~\ref{p:DeductionTheorem}) but ${\vdash} \neq {\vdash^s_\BCI}$ and so ${\vdash} \neq (\vdash^a)^s$.
Indeed we have $\f\circ\f\vdash[\f,\f]$ but as $\nvdash_\BCI\f$ we have $\f\circ\f\nvdash^s_\BCI[\f,\f]$.
\end{example}

\pagebreak

Let us consider some additional properties of symmetric relations between finite multisets and their interplay with the defining conditions of symmetric consequence relations. The following are some natural candidates:

\begin{itemize}
\item $\Gamma,\Delta\vdash \Gamma$. \hfill \GENREFL
\item $\Gamma\vdash \emptymultiset$. \hfill \TRASH
\item  If\/ $\Gamma\vdash \Delta$ and $\Delta,\Psi\vdash \Phi$, then $\Gamma,\Psi\vdash \Phi$. \hfill  \MULTICUT 
\item If\/ $\emptymultiset\vdash \Delta$ and $\Delta, \Psi\vdash \Phi$, then $\Psi\vdash \Phi$. \hfill \THMREM
\end{itemize}

The proofs of the following are straightforward, and left to the enthusiastic reader.

\begin{proposition}\label{p:OnCuts1}
Let ${\vdash}$ be a binary relation on finite multisets of formulas.
\begin{itemize}
\item[(i)] \REFL\ and \MONO\ entail \GENREFL
\item[(ii)] \GENREFL\ and \TRAN\ entail \MONO
\item[(iii)] \TRASH\ and \COMP\ entail \GENREFL
\end{itemize}
Thus in particular a relevant symmetric consequence relation is monotone iff it satisfies any of the following three conditions: \MONO, \GENREFL, or \TRASH.
\end{proposition}

\begin{proposition}\label{p:OnCuts2}
Let ${\vdash}$ be a binary relation on finite multisets of formulas.
\begin{itemize}
\item[(i)] \MULTICUT\ entails \TRAN\ and \THMREM\

\item[(ii)] \MULTICUT\ and \REFL\ entail \COMP

\item[(iii)] \TRAN\ and \COMP\ entail \MULTICUT 

\end{itemize}
Thus in particular any relevant symmetric consequence relation enjoys \MULTICUT\ and \THMREM. 
\end{proposition}

\section{Derivations in the Symmetric Case}

In the asymmetric case we started from the concrete, with the notion of a (relevant) tree proof, and moved to the abstract: in the symmetric case, we're in the process of doing the opposite. So far we've fixed the abstract notion of symmetric consequence, building on the asymmetric one, but that leaves the natural question of how to understand derivations of multiple conclusions from multiple premises, in our sense. To that end, let us fix some definitions and infer some immediate results indicating that our abstract notion of symmetric consequence captures an intuitive concrete one (relying, as always, on our assumptions of finitude).

\begin{definition}
A \emph{symmetric axiomatic system} $\AS$ in $A$ is a set of $\form$-consecutions of the form $\Gamma\rhd\Delta$, where $\Gamma,\Delta$ are finite multisets of formulas -- i.e., $\AS$ is a set of \emph{multiple conclusion} consecutions. As in the asymmetric case, if $\Gamma=\emptymultiset$ then the consecution $\Gamma\rhd\Delta$ is an axiom, and otherwise it is a rule of inference.
\end{definition}

\begin{definition}[Finitary derivation]\label{d:proof}
Let $\AS$ be a symmetric axiomatic system in $A$. A \emph{derivation} of $\Delta$ from $\Gamma$ in $\AS$ is a finite sequence $\Gamma = \Gamma_1,\dots, \Gamma_n \ge \Delta$ of finite multisets of formulas such that for every $1< i \leq n$, there is a rule $\Psi \rhd \Psi'\in\AS$, such that 
$$
\Psi \le \Gamma_{i}\qquad\qquad\text{and} \qquad\qquad  \Gamma_i  = (\Gamma_{i-1}\Msetminus \Psi) \uplus \Psi'
$$
A derivation of $\Delta$ is \emph{relevant} if $\Gamma_n = \Delta$. We say that $\Delta$ is \emph{(relevantly) derivable} from $\Gamma$ in $\AS$, and write $\Gamma\vdash^{(r)}_{\!\AS}\Delta$, if there is a (relevant) derivation of\/ $\Delta$ from $\Gamma$ in $\AS$.
\end{definition}

The following lemma supports the adequacy of the given definition.

\begin{lemma}\label{prelu}
Let $\AS$ be a symmetric axiomatic system in $A$. Then $\vdash^{r}_{\!\AS}$ is the least symmetric  consequence relation containing $\AS$ and  $\vdash_{\!\AS}$ is the least monotone symmetric  consequence relation containing $\AS$.
\end{lemma}

\begin{proof}
Now we proceed to check all the properties of  (monotone) symmetric consequence relations
\begin{description}
\item[\REFL:] This is clear from the sequence $\langle \Gamma_{1}\rangle$ where $\Gamma_{1}= \Gamma$.
\item[\COMP:] Given a (relevant) derivation $P=\langle \Gamma_{1},\dots,\Gamma_{n}\rangle $ of $\Delta$ from $\Gamma$ in $\AS$, it is easy to observe that the sequence $P'=\langle \Gamma_{1}\uplus\Pi,\dots,\Gamma_{n} \uplus\Pi\rangle $ is a derivation of $\Delta\uplus\Pi$ from $\Gamma\uplus\Pi$ (thanks to the fact that in our notion of proof, we can apply rules in an arbitrary context).

\item[\TRAN:] The relevant case is obvious: one must simply concatenate the relevant derivations of $\Delta$ from $\Gamma$ and of $\Psi$ from $\Delta$ (with having $\Delta$ there only once in the middle). The monotone case is a bit trickier as we have to extend all elements of the latter derivation by the formulas in the last step of the former derivation which are not in $\Delta$ (so the last step of the former derivation is the same as the first step of the latter one).

\item[\MONO:] This is a simple exercise left to the reader.
\end{description}

Now for the proof that $\vdash_{\!\AS}$ is the least such relation. Obviously $\AS \subseteq{\vdash^{r}_{\!\AS}}$. Consider 
a symmetric consequence relation~$\vdash$ such that $\AS\subseteq{\vdash}$, and we show that then ${\vdash^r_{\!\AS}}\subseteq {\vdash}$. Assume that $\Gamma\vdash^r_{\!\AS}\Delta$, i.e., there is a relevant derivation $\Gamma_1, \dots, \Gamma_n$ of $\Delta$ from $\Gamma$ in $\AS$. By induction, we can show that $\Gamma\vdash\Gamma_i$, and hence in particular $\Gamma\vdash\Delta$. The base case is settled with an appeal to \REFL. As to the induction step: we know that $\Gamma_{i+1} = (\Gamma_i\Msetminus\Psi)\uplus\Psi'$ for some rule $\Psi \rhd\Psi' \in\AS\subseteq {\vdash}$. Thus $\Psi\vdash\Psi'$, and so by \COMP\ also $\Psi\uplus(\Gamma_{i}\Msetminus\Psi)  \vdash \Psi'\uplus(\Gamma_i\Msetminus\Psi)  $, i.e., $\Gamma_i\vdash \Gamma_{i+1}$. As due to the induction assumption we know that $\Gamma\vdash\Gamma_i$, an application of \TRAN\ completes the proof. The monotone case is analogous: as we only know that $\Gamma_n \le \Delta$ we only prove that $\Gamma\vdash \Gamma_n$; here, however, the monotonicity yields the required result.
\end{proof}


\begin{example}
Recall \BCI\ from Section~\ref{sect:BCI-Example}.  Let's consider some provable and unprovable symmetric consecutions in this logic. As a very simple example, note that:
\[
    \f\to\p,\f\to\x,\f,\f \nvdash[\f,\p,\x]
\]
as while we can obtain assymetric proofs for $\f\to\p,\f\vdash\p$ and $\f\to\x,\f\vdash\x$, as we cannot find any further formulas from which to infer the $\f$ occurring among the premises. However, this is the case if we take:
$$\f\to\p,\f\to\x,\f,\f,\f \vdash[\f,\p,\x] $$

\noindent Furthermore, the relevance constraints mean that, in general, $\f,\p\nvdash^r[\f]$. Interestingly, while in the asymmetric case we have $\f\nvdash^r\f\to\f$, when we go symmetric we obtain something related, namely, $\f\vdash^r[\f\to\f,\f]$, for we can obtain, in the asymmetric setting, $\vdash^r\f\to\f$ and $\f\vdash^r\f$. This is related to the fact that, if we add~$\circ$, along with a constant formula $t$ expressing $\emptymultiset$ (i.e., something like a left identity in the sense of \cite{DunnH}) to the language, that $\f\to(\f\to\f)$ is not a valid implication, but we do have $\f\to((\f\to\f)\circ\f)$, as, $t\to(\f\to\f)$ and $\f\to\f$ are both valid, and thus so is $\f\to(t\circ\f)\to((\f\to\f)\circ\f)$. 


\end{example}

The next propositions spell out the relation between derivations and tree-proofs.

\begin{proposition}\label{p:trees}
Let $\AS$ be a single\=/conclusion axiomatic system, $\Gamma\vdash^r_{\!{\!\AS}}\Delta$, and $\varphi \in \Delta$. Then there are multisets of formulas $\Gamma^{\varphi}$ and $\Gamma^{r}$ such that $\Gamma^{\varphi }\uplus\Gamma^{r}=\Gamma$, $\Gamma^{\varphi}\vdash_{\!{\!\AS}}^{r}[\varphi]$ and $\Gamma^{r}\vdash^r_{\!{\!\AS}}\Delta\Msetminus [\f]$.
\end{proposition}

\begin{proof}

Let $\Gamma_{1},\dots,\Gamma_{n}$ be the assumed relevant derivation of $\Delta$ from $\Gamma$ in $\AS$. For each $i\leq n$, let $\Gamma_{i}=[\psi_{1}^{i},\dots,\psi_{k_{i}}^{i}]$ and note that without loss of generality we can assume  that the rule used in the $i$-th step of the proof is $[\psi_{1}^{i}, \dots, \psi_{p_i}^{i} ] \rhd\psi_{c_{i}}^{i+1}$ for some $p_i \leq k_{i}$ and $c_{i}\leq k_{i+1}$. Also note that $k_{i+1}=k_{i+1}- p_{i} + 1$ and there is a bijection $f$ between $\Gamma_{i+1}\Msetminus\lbrack\psi_{c_{i}}^{i+1}]$ and $\Gamma _{i}\Msetminus[\psi_{1}^{i}, \dots, \psi_{p_i}^{i} ]$ such that $\psi_{j} ^{i+1}=\psi_{f(j)  }^{i}$ whenever $c_{i}\neq j\leq k_{i+1}$. We construct the labelled graph $G$ with nodes $N=\{\langle i,j\rangle\mid i\leq n\text{ and }j\leq k_{i}\}$, where $\psi_{j}^{i}$ is the label of $\langle i,j\rangle$, and edges only between the following nodes:
\begin{enumerate}
\item[$\bullet$] \emph{rule edges}: $\langle i,k\rangle$ and $\langle{i+1},c_{i}\rangle$
for each $k\leq p_{i};$

\item[$\bullet$] \emph{non-rule edges}: $\langle i,f(j)\rangle$ and $\langle{i+1},j\rangle$ for $j\neq c_{i}$.
\end{enumerate}

It is easy to see that~$G$ is a forest (a disjoint union of trees). Let~$t$ be the subtree of~$G$ with root labeled by~$\f$ (if there are more such nodes, choose any), and let $\Gamma^{\varphi}$ be the multiset of all labels of leaves in $t$ which are not axioms. Then clearly $\Gamma^{\varphi}\le\Gamma$ and $t$ is \emph{almost} a relevant tree-proof of $\varphi$ from $\Gamma^{\varphi}$; all we have to do is to collapse nodes connected by non-rule edges.

Finally, let us by $t_i$ denote the multiset of labels occurring on the $i$-th level of $t$ (assuming that leaves are on the $1$st level and root on the $n$-th one) and set $\Gamma_{i}^{t} = \Gamma\Msetminus t_i$. Observe that $\Gamma_{1}^{t},\dots,\Gamma_{n}^{t}$ is \emph{almost} a proof of $\Gamma_{n}^{t}=\Gamma_{n}\Msetminus[\f]$ from $\Gamma_{1}^{t}$: we only need to remove each $\Gamma_{i}^{t}$ which equals its predecessor. Defining $\Gamma^{r}=\Gamma_{1}^{t}$ and observing that $\Gamma^{r} =\Gamma\Msetminus\Gamma^{\varphi}$ and $\Delta\Msetminus[\f]=\Gamma_{n}^{t}$ completes the proof.
\end{proof}

\begin{proposition}
Let $\vdash$ be a consequence relation (n.b.\ not a \emph{symmetric} consequence relation) with an axiomatic system $\AS$. Then $\AS$ is a single\=/conclusion axiomatic system for the symmetric consequence ${\vdash^s}$.
\end{proposition}

\begin{proof}
As clearly $\AS \subseteq {\vdash^s}$,  it suffices to show that $\Gamma \vdash^s \Delta$ implies $\Gamma\vdash^r_{\!{\!\AS}}\Delta$, which is easy to see.
\end{proof}

\begin{proposition}
A symmetric consequence relation $\vdash$ has a single-conclusion axiomatic system iff\/ ${\vdash} = {(\vdash^a)^s}$. Therefore monotone and contractive consequence relations always have single-conclusion axiomatic presentations.
\end{proposition}

\begin{proof}
Assume that  $\vdash$ has a single-conclusion axiomatic system; thanks to claim (ii) of Proposition~\ref{p:OnSymmetrization} we only have to prove that ${\vdash} \subseteq {(\vdash^a)^s}$. Assume first that $\Gamma\vdash \emptymultiset $ and we have to prove that $\Gamma \vdash^a \x$ for any theorem $\x$ of $\vdash^a$. As $\x$ is a theorem of $\vdash^a$, then we have $\vdash[\x]$ and so $\TRAN$ completes the proof. Second, assume that $\Gamma\vdash[\vectn{\chi}]$, and by $n$ uses of Proposition~\ref{p:trees} we obtain for each $i \leq n$ a multiset $\Gamma_i \le \Gamma$  such that  $\Gamma_i \vdash^a \x_i$ and  $\Gamma_1\uplus\dots\uplus\Gamma_n = \Gamma$; thus demonstrating that $\Gamma\mathrel{(\vdash^a)^s}[\vectn{\chi}]$.

The converse direction follows from the previous proposition and the final claims from the last part of Proposition~\ref{p:OnSymmetrization}.
\end{proof}

Using this proposition we know that the symmetric versions of Abelian logic and logic BCI introduced in Example~\ref{e:symAbel} have no single-conclusion axiomatic systems.



\section{Theories}\label{s:Theories}

Now that we have pictures of (relevant) symmetric consequence in abstract and concrete terms, the question of \emph{deductive closure}, and the properties of \emph{theories} is quite natural. The standard definition of this concept in the Tarskian setting is well known, but there are some complications which come up in our case. Let us work our way towards a natural definition, and show that it has some properties. 


Recall that we work with a fixed set of formulas $\form$. In the classical Tarskian setting (defined over potentially infinite sets of formulas, cf.\ Remark~\ref{r:TarskiCanBeRelevant}), a $\vdash$-theory is a deductively closed set of formulas (i.e., a set of formulas which contains all its consequences) and the set of all $\vdash$-theories forms a closure systems on the set of formulas which fully determines both the consequence relation and its symmetrization:
\begin{itemize}
\item $\Gamma\vdash \f$ iff $\f$ belongs to all $\vdash$-theories containing $\Gamma$ as a subset.
\item $\Gamma \vdash \Delta$ iff $\Delta$ is a subset of the $\vdash$-theory generated by $\Gamma$ (i.e., the set $\{\f\mid \Gamma\vdash\f\}$).
\end{itemize}

The problem with this approach for our setting is that the second condition clearly entails the \MONO\ condition. Further problems are caused by restriction to finite sets and by working with multisets rather than sets. Our guiding idea will be to see theories as particular \emph{sets of finite multisets of formulas} retaining as many nice properties of the Tarskian case as possible.  

As we have shown in the previous section, not every symmetric consequence relation arises from an asymmetric one; therefore we define all the notions in this section for the symmetric case, which can be retold for asymmetric relations via their symmetrizations. Let us start by observing that due to \REFL\ and \TRAN, every symmetric consequence relation $\vdash$ is a preorder on the set of finite multisets formulas and thus we can speak about $\vdash$\=/upsets, i.e., sets $T$ such that if $\Gamma\in T$ and $\Gamma\vdash \Delta$, then $\Delta\in T$. We say that a $\vdash$-upset $T$ is \emph{principal} if there is a finite multiset of formulas $\Gamma$ such that $T = \{\Delta\mid \Gamma\vdash\Delta\}$. 

Let us now define the notion of $\vdash$\=/\emph{theory} which will correspond to the notion of finitely generated theory in the Tarskian setting; we will refrain from defining a general notion of $\vdash$\=/\emph{theory} as it would be alien to our needs here. 

\begin{definition}\label{teoria}
Let\/ $\vdash$ be a symmetric consequence relation. A $\vdash$\=/\emph{theory} is a \emph{principal} $\vdash$\=/upset on the $\FMPow(\form)$ formulas, i.e., a set $\mathrm{Th}_{\vdash}(\Gamma) = \{\Delta\mid \Gamma\vdash\Delta\}$ for some finite multiset $\Gamma$. Let us by $\mathrm{Th}(\vdash)$ denote the set of all $\vdash$\=/theories respectively.
\end{definition}


We will denote theories by uppercase Latin letters $T$, $S$, $R$, etc. Also if  $\vdash$ is a  consequence relation, we speak about $\vdash$-theories instead of $\vdash^s$-theories and analogously for all upcoming notion in this section.

Obviously, $\mathrm{Th}_{\vdash}(\Gamma)$ is the least theory containing $\Gamma$ (this is an easy exercise left to the reader).\footnote{Information about theories in the Tarskian setting can be found in any standard book on algebraic logic, e.g., \cite{Font, CN}.} Furthermore, our notion of theory intuitively captures the idea that theories are deductively closed; it's just that the objects of deductive interest are themselves \emph{multisets} of formulas, and not just formulas, unlike in the standard setting. Despite this `type-lifting', though, as we'll see in the next pair of propositions, our notion of theory has a desirable property from the Tarskian setting, and collapses down to it when we build in the properties distinguishing relevant from Tarskian consequence relations. Consider the two bulleted conditions above.

\begin{proposition}
Let\/ $\vdash$ be a  symmetric consequence relation. Then for each finite multisets of formulas $\Gamma$ and $\Delta$ we have:
$$
\Gamma\vdash \Delta \qquad \text{iff}\qquad \Delta \in \bigcap \{T\in  \mathrm{Th}(\vdash) \mid \Gamma\in T\} \qquad \text{iff}\qquad \mathrm{Th}_\vdash(\Delta) \subseteq \mathrm{Th}_\vdash(\Gamma).
$$
\end{proposition}

\begin{proof}
Assume that $\Gamma\vdash \Delta$. If  $\Gamma\in T=\mathrm{Th}_{\vdash}(\Theta)$, we have that $\Theta\vdash \Gamma$ and by \ \TRAN \ it follows that $\Theta\vdash \Delta$, which means that $\Delta\in T$,  so $\Delta \in \bigcap \{T\in  \mathrm{Th}(\vdash) \mid \Gamma\in T\}$. Suppose further that $\Delta\vdash \Psi$, then, assuming that $\Delta \in \bigcap \{T\in  \mathrm{Th}(\vdash) \mid \Gamma\in T\}$, we have that $\Delta \in \mathrm{Th}_{\vdash}(\Gamma)$  (appealing to \REFL). Thus,  $\mathrm{Th}_\vdash(\Delta) \subseteq \mathrm{Th}_\vdash(\Gamma)$.

For the other direction of the sequence of implications, assume that $\mathrm{Th}_\vdash(\Delta) \subseteq \mathrm{Th}_\vdash(\Gamma)$.  If  $\Gamma\in T=\mathrm{Th}_{\vdash}(\Theta)$, by \ \REFL \ and our hypothesis, $\Gamma\vdash \Delta$, so by \ \TRAN\ it follows that  $\Delta\in T=\mathrm{Th}_{\vdash}(\Theta)$, and hence, $\Delta \in \bigcap \{T\in  \mathrm{Th}(\vdash) \mid \Gamma\in T\}$. Now if $\Delta \in \bigcap \{T\in  \mathrm{Th}(\vdash) \mid \Gamma\in T\}$, it follows that $\Gamma\vdash \Delta$, using \REFL.
\end{proof}

Recall that if $\vdash$ is monotone and contractive, we can see it as a binary relation on finite sets of formulas (instead of finite multisets of formulas), so theories can be seen as sets of finite sets of formulas and we can formulate and prove the following claim, roughly saying that the usual Tarskian theory is the union of our theory.

\begin{proposition}
Let\/ $\vdash$ be a monotone and contractive symmetric consequence relation. Then for each $\vdash$\=/\emph{theory} $T$ we have:
$$
T = \{\Gamma \mid \Gamma\subseteq \textstyle\bigcup T\}.
$$
\end{proposition}

\begin{proof} 
One inclusion is trivial, to show the latter assume that $T = \mathrm{Th}_{\vdash}(\Delta)$ and that $\Gamma\subseteq  \bigcup T$. As $\Gamma$ is finite, there are $\Delta_1, ..., \Delta_n \in T$ such that $\Gamma \subseteq \Delta_1 \cup \dots \cup \Delta_n$. Due to contraction we know that that $\Delta \vdash \Delta_1 \cup ... \cup \Delta_n$ and due to \REFL \ and \MONO, we know $\Delta_1 \cup ... \cup \Delta_n \vdash \Gamma$, which implies that $\Gamma \in T$.  
\end{proof}




Now that we have a collection of objects, we would love to find an algebraic structure on them: similarly to in the Tarskian case, we can find a rather natural one. In particular, $\FMPow(\form)$ can be seen as the domain of an \emph{Abelian monoid} with a neutral element $\emptymultiset$ and addition $\uplus$: we'll denote this structure by $\alg{Fm}^\#$. Recall that given an Abelian monoid $\alg{A} =  \tuple{A, +, 0}$ and an order $\leq$ on $A$ we say that $\tuple{\alg{A}, \leq}$ is an \emph{ordered} Abelian monoid if $a\leq b$ implies $a+c \leq b+c$ ($+$ is monotone w.r.t. $\leq$). Let's verify that we can find such a monoid structure on the class of theories:

\begin{theorem}\label{DRandTheories}
Let\/ $\vdash$ be a symmetric consequence relation. Then the algebra $\alg{Th}_\vdash =  \tuple{\mathrm{Th}(\vdash), +_\vdash, 0_\vdash}$ where, $0_\vdash = \mathrm{Th}_\vdash(\emptymultiset)$ and
$$
\mathrm{Th}_\vdash(\Gamma) +_\vdash \mathrm{Th}_\vdash(\Delta) = \mathrm{Th}_\vdash(\Gamma \uplus \Delta).
$$
is an Abelian monoid and the mapping $\mathrm{Th}_\vdash\colon \alg{Fm}^\# \to \alg{Th}_\vdash$ is a surjective homomorphism.

Furthermore $\tuple{\alg{Th}_\vdash, \subseteq}$ is an ordered Abelian monoid and, finally, $\vdash$ is monotone iff\/ $\mathrm{Th}_\vdash$ is a monotone mapping from $\tuple{\alg{Fm}^\#, \le}$ into $\tuple{\alg{Th}_\vdash, \subseteq}$. 
\end{theorem}

\begin{proof}
Let us first check that $+_\vdash$ is well-defined. To this end, consider multisets of formulas $\Gamma,\Gamma',\Delta,\Delta'$ such that $\mathrm{Th}_\vdash(\Gamma)=\mathrm{Th}_\vdash(\Gamma')$ and $\mathrm{Th}_\vdash(\Delta)=\mathrm{Th}_\vdash(\Delta')$. In particular, $\Gamma'\vdash \Gamma$ and $\Delta'\vdash \Delta$ and so by Compatibility and Transitivity $\Gamma'\uplus\Delta'\vdash \Gamma\uplus\Delta$, whence $\mathrm{Th}_\vdash(\Gamma\uplus\Delta)\subseteq\mathrm{Th}_\vdash(\Gamma'\uplus\Delta')$. The other inclusion is proved analogously. 

With this fact about $+_\vdash$, it is easy to see  that $\tuple{\mathrm{Th}(\vdash), +_\vdash, 0_\vdash}$ is an Abelian monoid. Commutativity and associativity of $+_\vdash$ are trivial by the corresponding properties of $\uplus$ as a multiset operation. Similarly,  $0= \mathrm{Th}_\vdash(\emptymultiset)$ is an identity (i.e., $\mathrm{Th}_\vdash(\Gamma\uplus\emptymultiset)= \mathrm{Th}_\vdash(\Gamma)$) since by properties of multisets, $\Gamma\uplus\emptymultiset = \Gamma$. The fact that the mapping $\mathrm{Th}_\vdash\colon \alg{Fm}^\# \to \alg{Th}_\vdash$ is a surjective homomorphism is similarly easy to see.

Let us check now that $\tuple{\alg{Th}_\vdash, \subseteq}$ is an ordered Abelian monoid. Assume that $\mathrm{Th}_\vdash(\Gamma) \subseteq \mathrm{Th}_\vdash(\Delta)$. Suppose further that $\Psi \in \mathrm{Th}_\vdash(\Gamma\uplus \Theta)$, so $\Gamma\uplus \Theta \vdash \Psi$. By \ \REFL  \ and our hypothesis, $\Gamma \vdash \Delta$. Now, by \ \COMP \/, $\Gamma \uplus \Theta\vdash \Delta \uplus \Theta$.  Hence, $\Delta \uplus \Theta \vdash \Psi$ (i.e., $\Psi \in \mathrm{Th}_\vdash(\Delta \uplus \Theta)$) as desired by \TRAN.

Suppose now that $\vdash$ is monotone. We want to show that $\mathrm{Th}_\vdash$ is a monotone mapping. Assume that $\Gamma \le \Delta$, so if $\Psi \in \mathrm{Th}_\vdash(\Gamma)$, we have that $\Gamma \vdash \Psi$ and by \MONO,  we have that $\Gamma \uplus  (\Delta \Msetminus \Gamma) \vdash \Psi$, so $\Psi \in \mathrm{Th}_\vdash(\Delta)$ as desired. On the other hand, suppose that  $\mathrm{Th}_\vdash$ is a monotone mapping. Assume further that $\Gamma \vdash \Psi$, so  $\Psi \in \mathrm{Th}_\vdash(\Gamma)$, which implies that $\Psi \in \mathrm{Th}_\vdash(\Gamma \uplus \Phi)$ for any $\Phi$ by monotonicity of the mapping $\mathrm{Th}_\vdash$ and the fact that $\Gamma \le \Gamma \uplus \Phi$. This gives us the monotonicity of $\vdash$.
\end{proof}

Similarly, it's desirable that we be able to define a congruence on $\alg{Fm}^\#$ using $\vdash$, which has other nice properties. This is straightforward: consider a symmetric consequence relation $\vdash$ and define a binary relation $\dashv\vdash$ of finite multisets of formulas as
$$
\Gamma\dashv\vdash \Delta \qquad\text{iff}\qquad \Gamma\vdash \Delta\ \text{ and }\ \Delta\vdash \Gamma
$$

\begin{theorem}
Let $\vdash$ be a symmetric consequence relation on $\form$. Then
\begin{itemize}
\item[(i)] the relation $\dashv\vdash$ is a congruence on $\alg{Fm}^\#$,
\item[(ii)] the tuple  $\tuple{\alg{Fm}^\#_\vdash, \leq_\vdash}$ is an ordered Abelian monoid, where $\alg{Fm}^\#_\vdash$ is the $\dashv\vdash$-quotient of $\alg{Fm}^\#$, and
$$
[\Gamma]_{\dashv\vdash} \leq_\vdash [\Delta]_{\dashv\vdash} \qquad \text{iff} \qquad \Gamma \vdash  \Delta,
$$

\item[(iii)]
the mapping $\mathrm{Th}'_\vdash \colon \tuple{\alg{Fm}^\#_\vdash, \leq_\vdash} \to \tuple{\alg{Th}_\vdash, \subseteq}$ defined as
$$
\mathrm{Th}'_\vdash([\Gamma]_{\dashv\vdash}) = \mathrm{Th}_\vdash(\Gamma)
$$
is an antitone surjective homomorphism.
\end{itemize}
\end{theorem}
\begin{proof}
 (i): By \ \REFL \/ and \TRAN, the relation  $\dashv\vdash$ is an equivalence. Suppose that now $\Gamma \dashv\vdash \Gamma'$ and $\Delta \dashv\vdash \Delta'$. We wish to show then that $\Gamma \uplus \Delta \dashv\vdash \Gamma' \uplus \Delta'$. By \COMP, we have that $\Gamma \uplus \Delta \vdash \Gamma' \uplus \Delta$ and $\Gamma' \uplus \Delta \vdash \Gamma' \uplus\Delta'$, so by \ \TRAN \/ we have that  $\Gamma \uplus \Delta \vdash \Gamma' \uplus\Delta'$. Analogously, $  \Gamma' \uplus\Delta' \vdash \Gamma \uplus \Delta$.

(ii): It is not difficult to see that $\tuple{\alg{Fm}^\#_\vdash, \leq_\vdash}$ is an Abelian monoid. We will show then that   $[\Gamma]_{\dashv\vdash}\leq_\vdash [\Delta]_{\dashv\vdash}$ implies that $[\Gamma \uplus \Theta]_{\dashv\vdash} \leq_\vdash [\Delta \uplus \Theta]_{\dashv\vdash}$. This simply follows as a consequence of \COMP \/ and the definition of the order.

(iii):We will leave checking that the mapping is a surjective homomorphism as 
an exercise for the reader. Finally, let us establish the monotonicity property. Suppose that $[\Gamma]_{\dashv\vdash}\leq_\vdash [\Delta]_{\dashv\vdash}$. We want to show that $\mathrm{Th}_\vdash(\Delta)\subseteq \mathrm{Th}_\vdash(\Gamma) $. If $\Psi \in \mathrm{Th}_\vdash(\Delta)$, then $\Delta\vdash \Psi$ but since $\Gamma \vdash  \Delta$, by \ \TRAN \/, we have that $\Gamma\vdash \Psi$, so $\Psi \in \mathrm{Th}_\vdash(\Gamma)$.
\end{proof}


\section{Conclusion}

In this paper we have begun to investigate external consequence relations capturing some of the properties sought in relevant logics, particularly that the premises of a consecution ought to be \emph{used} in obtaining the conclusion. Starting from some intuitive desiderata for how such relations ought to behave, we have fixed plausible definitions for asymmetric and symmetric consequence relations, and proved a number of properties of these indicating that they have interesting properties, and allow for a finer grained analysis of logical consequence than the usual, Tarskian, paradigm allows. We've considered relevant consequence relations in abstract terms, with appeal to concrete presentations of logics, taking some relevant and substructural logics as our paradigm examples, and have pinned down a natural definition of theory appropriate to our aims. This work builds on previous work on liberalizing the standard definition of consequence in ways appropriate to the motivations of various non-classical logics, and adapts it to the relevant case. 

This paper is not the final word on the topics we've investigated, and there remain a number of promising avenues for future work. Hopefully by the cogency of our definitions and theorems, we will have convinced the reader that there is a rich seam of logical work to be done here, which has clear and compelling philosophical motivations. In the blunt words of \cite{Weinersmith}: This paper is not meant to be comprehensive or conclusive, but only a first step, taken in order to claim priority after someone else does the hard work.

\section*{Acknowledgements}

Libor B\v ehounek and Petr Cintula were supported by the Czech Science Foundation project GA22-01137S. Petr Cintula was also supported by RVO 67985807. Andrew Tedder was supported by fellowship funding from the Alexander von Humboldt foundation. Guillermo Badia is supported by  the Australian Research Council grant DE220100544.

\end{document}